\documentclass[11pt]{article}

\usepackage{amsmath}
\usepackage{amsfonts}
\usepackage{amsthm}
\usepackage{amssymb}
\usepackage{mathrsfs}
\usepackage{verbatim}
\usepackage{enumerate}
\usepackage{graphicx}
\usepackage{longtable}
\usepackage[all,cmtip]{xy}
\usepackage{upref}
\usepackage[hang,small,bf]{caption}
\usepackage{marginnote}
\usepackage{a4wide}
\usepackage{mathtools,tensor}
\usepackage{xfrac}
\usepackage{color}
\usepackage{bm}
\usepackage{hyperref}
\usepackage{nomencl}
\usepackage{xtab,cuted}
\usepackage[normalem]{ulem}

\newtheorem{thm}{Theorem}[section]
\newtheorem{prp}[thm]{Proposition}
\newtheorem{lem}[thm]{Lemma}
\newtheorem{cor}[thm]{Corollary}

\newtheorem*{con*}{Conjecture 2}
\theoremstyle{definition}

\newtheorem{rmk}[thm]{Remark}

\numberwithin{equation}{section}

\newcommand{\T}{{\mathbb T}}
\newcommand{\Z}{{\mathbb Z}}

\newcommand{\R}{{\mathbb R}}
\newcommand{\N}{{\mathbb N}}

\renewcommand{\AA}{{\mathcal A}}
\newcommand{\PP}{{\mathcal P}}

\newcommand{\UU}{{\mathcal U}}
\newcommand{\HH}{{\mathcal H}}

\renewcommand{\H}{\mathrm{H}}

\newcommand{\HZ}{\mathrm{HZ}}
\newcommand{\Crit}{\mathrm{Crit\,}}

\newcommand{\id}{\mathrm{id}}

\newcommand{\di}{{\mathrm d}}

\newcommand{\rk}{\mathrm {rk\,}}

\newcommand{\sk}{{\mathrm{sk}}}
\newcommand{\SH}{{\mathrm{SH}}}
\newcommand{\HF}{{\mathrm{HF}}}
\newcommand{\CF}{{\mathrm{CF}}}

\newcommand{\sys}{\mathrm{sys}}
\newcommand{\bir}{\mathrm{bir}}
\newcommand{\diff}{\mathrm{d}}
\newcommand{\injrad}{\mathrm{injrad}}
\newcommand{\xx}{\bm{x}}
\newcommand{\yy}{\bm{y}}
\newcommand{\vv}{\bm{v}}

\newcommand{\zz}{\bm{z}}

\newcommand{\p}{\partial}
\newcommand{\into}{\hookrightarrow}
\newcommand{\x}{\times}

\newcommand{\one}
{{{\mathchoice \mathrm{ 1\mskip-4mu l} \mathrm{ 1\mskip-4mu l}
\mathrm{ 1\mskip-4.5mu l} \mathrm{ 1\mskip-5mu l}}}}

\newcommand{\beq}{\begin{equation}}
\newcommand{\beqn}{\begin{equation}\nonumber}
\newcommand{\eeq}{\end{equation}}
\newcommand{\bea}{\begin{equation}\begin{aligned}}
\newcommand{\bean}{\begin{equation}\begin{aligned}\nonumber}
\newcommand{\eea}{\end{aligned}\end{equation}}

\AtEndDocument{\bigskip{\footnotesize
  \textsc{Ruprecht-Karls-Universit\"at Heidelberg, Mathematisches Institut, Im Neuenheimer Feld 205, 69120 Heidelberg, Germany} \par  
  \textit{E-mail address}: \texttt{\href{mailto:gbenedetti@mathi.uni-heidelberg.de}{gbenedetti@mathi.uni-heidelberg.de}} \par
  \addvspace{\medskipamount}
\noindent\textsc{Seoul National University, Department of Mathematical Sciences, Research Institute in Mathematics, Gwanak-Gu, 
	Seoul 08826, South Korea} \par  
  \textit{E-mail address}: \texttt{\href{mailto:jungsoo.kang@snu.ac.kr}{jungsoo.kang@snu.ac.kr}} \par
}}

\title{Relative Hofer--Zehnder capacity\\ and positive symplectic homology}
\author{Gabriele Benedetti and Jungsoo Kang\\ \\
(with an appendix by Alberto Abbondandolo and Marco Mazzucchelli)}
\date{}
\begin{document}
\maketitle
\begin{center}\textit{\large Dedicated to Claude Viterbo on the occasion of his sixtieth birthday}\end{center}
\medskip
\begin{abstract}
We study the relationship between a homological capacity $c_{\mathrm{SH}^+}(W)$ for Liouville domains $W$ defined using positive symplectic homology and the existence of periodic orbits for Hamiltonian systems on $W$: If the positive symplectic homology of $W$ is non-zero, then the capacity yields a finite upper bound to the $\pi_1$-sensitive Hofer--Zehnder capacity of $W$ relative to its skeleton and a certain class of Hamiltonian diffeomorphisms of $W$ has infinitely many non-trivial contractible periodic points. En passant, we give an upper bound for the spectral capacity of $W$ in terms of the homological capacity $c_{\mathrm{SH}}(W)$ defined using the full symplectic homology. Applications of these statements to cotangent bundles are discussed and use a result by Abbondandolo and Mazzucchelli in the appendix, where the monotonicity of systoles of convex Riemannian two-spheres in $\mathbb R^3$ is proved.
\end{abstract}

\section{Periodic orbits for Hamiltonian systems}\label{sec:intro}
In this paper, we prove new existence and multiplicity results for periodic orbits of Hamiltonian systems on Liouville domains using positive symplectic homology. We present our results in Section \ref{s:results}, and, in order to put them into context, we give in this first section a brief (and incomplete) account of previous work on the existence of periodic orbits for Hamiltonian systems that will be relevant to our work. For more details, we recommend the excellent surveys \cite{CHLS,Gin05,GG15}.  
\subsection{From calculus of variation to pseudoholomorphic curves}
From the end of the Seventies to the beginning of the Nineties of last century a series of tremendous advances has been made in understanding the existence of periodic solutions for smooth Hamiltonian systems in $\R^{2n}$ with standard symplectic form $\omega_{\R^{2n}}$ using variational techniques. Classically, two settings have been considered.

In the non-autonomous setting, one may consider a smooth Hamiltonian function $H$ which is periodic in time (say with period one) and compactly supported, and study the set $\mathcal P_m(H)$ of periodic orbits of the Hamiltonian vector field of $H$ with some integer period $m\in\N$. This is equivalent to looking at periodic points of the Hamiltonian diffeomorphism obtained as the time-one map of the Hamiltonian flow of $H$. The set $\mathcal P_m(H)$ surely contains trivial orbits, namely those which are constant and contained in the zero-set of the Hamiltonian, and the question is if there are non-trivial elements in this set. In 1992 Viterbo  proved the following remarkable result.

\begin{thm}[Viterbo, \cite{Vit92}]\label{t:V}
Let $\varphi\colon\R^{2n}\to\R^{2n}$ be a compactly supported Hamiltonian diffeomorphism different from the identity. Then, $\varphi$ admits a non-trivial one-periodic point and infinitely many distinct non-trivial periodic points. 
\end{thm} 

In the autonomous setting, the Hamiltonian function $H$ is coercive and independent of time. Thus, it is meaningful to study the set of periodic orbits having arbitrary real period and lying on a regular energy level $\Sigma$ of $H$, which is compact by coercivity. Let $j_0\colon\Sigma\into\R^{2n}$ be the inclusion of $\Sigma$, and denote the set of periodic orbits of $H$ on $\Sigma$ by $\mathcal P(\Sigma)$. This notation is justified since periodic orbits of $H$ on $\Sigma$ correspond to closed characteristics of $\Sigma$, i.e.~embedded circles in $\Sigma$ which are everywhere tangent to the characteristic distribution $\ker(j_0^*\omega_{\R^{2n}})$. In 1987 Viterbo \cite{Vit87} proved the existence of a closed characteristic for a large class of hypersurfaces $\Sigma$ in $\R^{2n}$, namely those of contact type (see \cite{HZ11} for the definition), thus confirming the Weinstein conjecture in $\R^{2n}$ \cite{Wei79}. This considerably extends previous results by Weinstein \cite{Wei78} and Rabinowitz \cite{Rab78} for hypersurfaces being convex or starshaped, two properties which are not invariant by symplectomorphisms.

Soon after, it was understood that Viterbo's result is a manifestation of more general phenomena which go under the name of nearby existence and almost existence. In order to explain these phenomena, we consider a thickening of $j_0\colon\Sigma\to\R^{2n}$, namely an embedding $j\colon(-\varepsilon_0,\varepsilon_0)\times\Sigma\to\R^{2n}$ such that $j(0,\cdot)=j_0$. We denote $\Sigma_\varepsilon=j(\{\varepsilon\}\times\Sigma)$ for $\varepsilon\in(-\varepsilon_0,\varepsilon_0)$. For a thickening of $j_0$, the nearby existence theorem \cite{HZ87} ensures that for a sequence of values $\varepsilon_k\to 0$, there holds $\mathcal P(\Sigma_{\varepsilon_k})\neq\varnothing$. On the other hand, the almost existence theorem \cite{Str90} yields the stronger assertion that $\mathcal P(\Sigma_\varepsilon)\neq\varnothing$ for almost every $\varepsilon\in(-\varepsilon_0,\varepsilon_0)$. Viterbo's result follows from either of these two theorems noticing that a contact, or more generally stable \cite{HZ11}, hypersurface admits a thickening such that the characteristic foliation of $\Sigma$ is diffeomorphic to the one of $\Sigma_\epsilon$ for all $\varepsilon\in(-\varepsilon_0,\varepsilon_0)$. A classical argument by Hofer and Zehnder \cite{HZ11} proves the almost existence theorem by showing the finiteness of the Hofer--Zehnder capacity for bounded domains of $\R^{2n}$ of which we now recall the definition. If $U\subset \R^{2n}$ is a domain possibly with boundary, let $\mathcal H(U)$ be the set of Hamiltonians $H\colon U\to(-\infty,0]$ that vanish outside a compact set of $U\setminus\p U$ and achieve their minimum on an open set of $U$. The Hofer--Zehnder capacity of $U$ is defined by
\[
c_{\HZ}(U):=\sup_{H\in\mathcal H(U)}\big\{-\min H\ \big|\ \mathcal P_{\leq 1}(H)=\mathrm{Crit}(H)\big\}\in(0,\infty],
\]
where we used the notation $\mathrm{Crit}(H)$ for the set of critical points of $H$ and $\mathcal P_{\leq t}(H)$ for the set of periodic orbits of the Hamiltonian vector field of $H$ with period at most $t$.
\medskip

Much of research in Hamiltonian dynamics in the last thirty years has been driven by considering the two settings described above for general symplectic manifolds $(M,\omega)$. For instance, one can ask if the nearby or almost existence theorems hold for a given hypersurface $\Sigma\subset M$. To this end, the Hofer--Zehnder capacity, whose definition extends verbatim to open subsets of $M$, still plays a central role as the almost existence theorem holds for hypersurfaces contained in a domain $U\subset M$ of finite Hofer--Zehnder capacity \cite{MS05}. It is also sometimes useful to consider a refined quantity, called the $\pi_1$-sensitive Hofer--Zehnder capacity, given by 
\[
c_{\HZ}^o(U):=\sup_{H\in\mathcal H(U)}\big\{-\min H\ \big|\ \mathcal P^o_{\leq 1}(H)=\mathrm{Crit}(H)\big\}\in(0,\infty],
\]
where $\mathcal P^o_{\leq 1}(H)$ is the set of elements in $\mathcal P_{\leq 1}(H)$ which are contractible in $M$. Clearly $c_{\HZ}(U)\leq c_{\HZ}^o(U)$ and finiteness of $c_{\HZ}^o(U)$  implies the almost existence theorem with the additional information that the periodic orbits we find are contractible in $M$. However, verifying finiteness of $c_{\HZ}(U)$ or $c_{\HZ}^o(U)$ is usually highly non-trivial and there are also many examples where these quantities are infinite \cite{Ush12}.

The first relevant class of manifolds one encounters outside $\R^{2n}$ is given by cotangent bundles $T^*Q$ over closed manifolds $Q$ endowed with the standard symplectic form $\omega_{T^*Q}$. Besides their importance as phase spaces in classical mechanics, they give local models for Lagrangian submanifolds of abstract symplectic manifolds by the Weinstein neighborhood theorem. This implies that if $Q$ admits a Lagrangian embedding into $\R^{2n}$ (for instance $Q$ is a torus), then any bounded subset $U\subset T^*Q$ can be symplectically embedded into a bounded subset of $\R^{2n}$ and hence the almost existence theorem holds for $U$, see for instance \cite[Proposition 1.9]{Iri14}. For general manifolds $Q$, Hofer and Viterbo \cite{HV88} could carry over the variational setting employed in $\R^{2n}$ by exploiting the fact that the fibers of $T^*Q$ are linear spaces to prove the following remarkable result (see the work of Asselle and Starostka \cite{AS20} for a simplified approach).
\begin{thm}[Hofer and Viterbo, \cite{HV88}]\label{t:HV}
The nearby existence theorem holds for hypersurfaces $\Sigma\subset T^*Q$ which bound a compact region containing the zero-section in its interior.
\end{thm}

To go beyond cotangent bundles, one leaves the classical variational approach and enters the theory of pseudoholomorphic curves initiated by Gromov in \cite{Gro85}. For closed symplectic manifolds, the non-autonomous setting is governed by the Arnold conjecture \cite{Arn86}, giving a lower bound on the number of fixed points of Hamiltonian diffeomorphisms, and by the Conley conjecture, asserting that for a large class of symplectic manifolds every Hamiltonian diffeomorphism has infinitely many periodic points. These conjectures led to a tremendous development in the field and have been settled by now in very general forms, see e.g.~\cite{Flo88,Flo89,HS95,Ono95,LO96,LT98,FO99} for the Arnold conjecture and \cite{Gin10,GG15} for the Conley conjecture.

Pseudoholomorphic curves have also been employed in the autonomous setting starting from the pioneering work of Hofer and Viterbo \cite{HV92}, where they are used to give an upper bound on $c_{\HZ}^o$ of certain symplectic manifolds. Their approach has been further developed using Gromov--Witten invariants in \cite{Lu98,LT00,Mac04,Lu06,Ush11}.
\medskip

For both the non-autonomous and the autonomous setting on general symplectic manifolds, a key role is played by the notion of displaceability and of spectral invariants. Recall that a domain $U$ is displaceable in $M$ if there exists a compactly supported Hamiltonian diffeomorphism $\phi_K^1$ of $M$ such that $\phi_K^1(U)\cap U=\varnothing$. Quantitatively speaking, one can define the displacement energy of $U$ by
\[
e(U):=\inf \big\{ \Vert K\Vert\ |\ \phi_K^1(U)\cap U=\varnothing\big\},
\]
where $\Vert K\Vert$ is the Hofer norm of $K$, see \cite{HZ11}. Finiteness of $e(U)$ is equivalent to displaceability of $U$. If $M$ is either closed or convex at infinity, one defines the spectral capacity of $U$ by
\[
c_\sigma(U):=\sup\{\sigma(H)\},
\]
where the supremum is taken over all Hamiltonian functions compactly supported in $U\setminus\p U$ and $\sigma(H)$ represents the spectral invariant of $H$ associated with the fundamental class in Floer homology \cite{Vit92,Sch00,Oh,FS07}. The chain of inequalities 
\begin{equation}\label{e:sigmae}
c_{\HZ}^o(U)\leq c_\sigma(U)\leq e(U)
\end{equation}
holds by \cite{Ush11,Sug19} generalizing work contained in \cite{Vit92,Sch00,FGS05,Schl06,FS07,HZ11} (see also \cite{Gur08} for an extension of this argument to some open, non-convex manifolds). One significant implication of finiteness of $c_\sigma(U)$ is that every Hamiltonian diffeomorphism $\varphi$ compactly supported in $U\setminus \p U$ has finite spectral norm and thus carries infinitely many non-trivial contractible periodic points. Since a bounded domain in $\R^{2n}$ is displaceable, the latter inequality in \eqref{e:sigmae} generalizes Theorem \ref{t:V}.
 
Thus, displaceability gives a remarkable criterion to show that $c_\sigma(U)$ and hence $c_\HZ^o(U)$ are finite. In many cases of interest, however, either it is difficult to prove that a subset is displaceable or it can be shown that many subsets are not displaceable even by topological reasons (see the notion of stable displacement for a fix of this second difficulty \cite{Schl06}). A precious tool to study periodic orbits in these cases is represented by symplectic homology, see e.g.~\cite{CGK04}. It was originally defined for $\R^{2n}$ by Floer and Hofer \cite{FH94} and for more general manifolds by Cieliebak, Floer and Hofer \cite{CFH95}, and further developed by Viterbo in \cite{Vit99}. In the next subsection, we discuss how symplectic homology can help us in finding periodic orbits on Liouville domains, an important class of symplectic manifolds.

\subsection{Capacities and symplectic homology for Liouville domains}\label{ss:HZ}

A Liouville domain $(W,\lambda)$ is a compact manifold $W$ with boundary $\p W$ such that the exterior derivative $\di\lambda$ of the one-form $\lambda$ is symplectic and the Liouville vector field $Y$ on $W$ characterized by $\di\lambda(Y,\cdot)=\lambda$ points outwards along $\p W$.  The one-form $\lambda|_{\p W}$ restricted to $\p W$ is a contact form and we denote by $\mathrm{spec}(R_\lambda)$ the set of periods of periodic orbits of the Reeb vector field $R_\lambda$ of $(\p W,\lambda|_{\p W})$. It is a nowhere dense closed subset of $\R$. 

We define the skeleton $W_\sk$ of the Liouville domain to be the set of $\alpha$-limits of the flow $\phi_Y^t$ of $Y$ on $W$, i.e.
\[
W_\sk:=\bigcap_{t<0}\phi_Y^{t}(W).
\]
The complement $W\setminus W_\sk$ is symplectomorphic to the negative half of the symplectization of $(\p W,\lambda|_{\p W})$ via the map 
\[
\Psi\colon\p W\x(0,1]\to W,\qquad (x,r)\mapsto \phi_Y^{\log r}(x)
\]
where $\Psi^*\lambda=r\lambda|_{\p W}$, and we can complete $(W,\di\lambda)$ by attaching a cylindrical end 
\begin{equation}\label{eq:completion}
\widehat W:= W\cup \big(\p W\x [1,\infty)\big)	
\end{equation}
using $\Psi$ and by setting $\hat\lambda=r\lambda|_{\p W}$ on the cylindrical end.

The simplest examples of Liouville domains are starshaped domains in $\R^{2n}$. In this case, $Y$ is the radial vector field and the skeleton is a single point. Other examples are fiberwise starshaped domains in the cotangent bundle $T^*Q$ over a closed manifold $Q$ with the canonical one-form. In this case, $Y$ is the fiberwise radial vector field and the skeleton is the zero-section. Moreover, $\R^{2n}$ and $T^*Q$ are exactly the completions of the two examples we just described. 
\medskip

The symplectic homology $\SH(W;\alpha)$ of $(W,\lambda)$ in the free-homotopy class $\alpha\in[S^1,W]$ is, roughly speaking, generated by periodic orbits of the Reeb vector field $R_\lambda$ and, when $\alpha=\one$ is the class of contractible loops, also by a Morse complex for $(W,\p W)$. We refer to Section \ref{sec:Floer_homology} for the precise definition of symplectic homology. One can also construct the positive symplectic homology $\SH^+(W;\alpha)$ which is generated just by the periodic Reeb orbits and not by the Morse complex, so that $\SH^+(W;\alpha)=\SH(W;\alpha)$ for $\alpha\neq\one$. 

Symplectic homology carries a natural filtration given by periods of Reeb orbits. Using this filtration, several kinds of symplectic capacities can be constructed. One of them is $c_{\SH}(W)\in(0,\infty]$ which reads off the minimal filtration level such that the Morse complex of $(W,\p W)$ is annihilated in $\SH(W;\one)$, see \eqref{eq:c_SH}, \cite[Section 5.3]{Vit99}, \cite{H00}, and \cite[Proposition 3.5]{GS18}. It is finite if and only if $\SH(W;\one)$ vanishes. We have the inequalities
\begin{equation}\label{e:HZ-SH}
c^{\,o}_{\HZ}(W)\leq c_{\SH}(W)\leq e(W),
\end{equation}
where the displacement energy is taken with respect to the completion of $W$. The former inequality is due to Irie \cite{Iri14}. The latter is due to Borman--McLean \cite[Theorem 1.5(ii)]{BM14}, Kang \cite[Corollary A1]{Kan} and Ginzburg--Shon \cite[Corollary 3.9]{GS18}, see also \cite[Corollary 3.10]{GS18} for an inequality involving the stable displacement energy, and \cite[Theorem 2.2]{Su16} for a similar inequality in wrapped Floer homology. We complement \eqref{e:HZ-SH} and \eqref{e:sigmae} by the following theorem which relies on Irie's idea in \cite{Iri14}.
\begin{thm}\label{t:cgamma}
For a Liouville domain $(W,\lambda)$, there holds $c_\sigma(W)\leq c_{\SH}(W)$, and thus
\[
c^{\,o}_{\HZ}(W)\leq c_\sigma(W) \leq c_{\SH}(W)\leq e(W).
\] 
In particular, if $\SH(W;\one)=0$, then every Hamiltonian diffeomorphism supported in $W\setminus\p W$ admits a non-trivial fixed point and infinitely many distinct non-trivial periodic points.
\end{thm}

We see that vanishing of $\SH(W;\one)$, which is the case e.g.~if $W$ is displaceable in its completion, immediately implies the almost existence theorem for contractible orbits and the statement of Theorem \ref{t:V} for $W$. This line of attack is however not suited to study finiteness of the Hofer--Zehnder capacity for bounded fiberwise starshaped domains $W$ in cotangent bundles $T^*Q$, since $\SH(W;\one)$ does not vanish in this case. Indeed by \cite{Vit96,SW06,AS06,AS14,AS15,Abo15}, there exists Viterbo's isomorphism 
\begin{equation}\label{e:viterbo-iso}
\H(\mathcal L_\alpha Q) \cong \SH(W;\alpha)
\end{equation}
over $\Z_2$-coefficients (which is sufficient for our purpose), where $\mathcal L_\alpha Q$ denotes the free-loop space of $Q$ in the class $\alpha\in[S^1,Q]\cong[S^1,W]$.
However, symplectic homology can still be effectively used to prove the almost existence theorem on cotangent bundles, and we describe two instances how this has been done.
\begin{enumerate}[(a)]
\item For a large class of closed manifolds $Q$ including those for which the Hurewicz map $\pi_2(Q)\to \H_2(Q;\Z)$ is non-zero, Albers, Frauenfelder, and Oancea \cite{AFO17} generalized an idea by Ritter in the simply connected case \cite[Corollary 8]{Rit09} and showed that the symplectic homology of $W\subset T^*Q$ twisted by some local coefficients vanishes. The capacity $c_{\SH}(W)$ using this twisted version still gives an upper bound for $c_{\HZ}^o(W)$.

On the other hand, Frauenfelder and Pajitnov \cite{FP17} considered the $S^1$-equivariant version $\H^{S^1}(\mathcal L_\one Q)\cong \SH^{S^1}(W;\one)$ of isomorphism \eqref{e:viterbo-iso} with $\alpha=\one$ and rational coefficients. Generalizing an approach due to Viterbo \cite{Vit97}, they observe that when $Q$ belongs to the class of rationally inessential manifolds, for instance when $Q$ is simply connected, then $\H^{S^1}(\mathcal L_\one Q)$ with rational coefficients vanishes by Goodwillie's theorem and gives an upper bound on $c_{\HZ}^o(W)$ by a finite capacity coming from $\SH^{S^1}(W;\one)$ similarly as above.
\item The total homology $\bigoplus_{\alpha}\SH(W;\alpha)$ admits a ring structure with unit $e$. Irie \cite{Iri14} proved that $c_{\HZ}(W)$ is finite if there exists a free-homotopy class $\alpha$ different from $\one$ such that $e=x\ast y$ for some $x\in \SH(W;\alpha)$ and $y\in \SH(W;\alpha^{-1})$. Moreover using the fact that  $\ast$ corresponds to the Chas--Sullivan product in $\bigoplus_\alpha\H_*(\mathcal L_\alpha Q)$ through the Viterbo isomorphism, he showed that such condition is fulfilled when the evaluation map $\mathcal L_\alpha Q\to Q$, $x\mapsto x(0)$ has a section.
\end{enumerate}

Like the inequality $c_\HZ^o(W)\leq c_\SH(W)$, also the inequality $c_\sigma(W)\leq c_{\SH}(W)$ established in Theorem \ref{t:cgamma} holds when one twists the coefficients as in \cite{AFO17}. We obtain therefore the following result in the spirit of Theorem \ref{t:V}.
\begin{cor}
Let $Q$ be a closed manifold such that the Hurewicz map $\pi_2(Q)\to \H_2(Q;\Z)$ is non-zero. Then, every compactly supported Hamiltonian diffeomorphism on $T^*Q$ has a non-trivial one-periodic point and infinitely many distinct non-trivial periodic points. \hfill\qed
\end{cor}
The approach coming from inequality \eqref{e:HZ-SH} and from (a) above are based on the vanishing of symplectic homology. One can ask if knowledge on the positive symplectic homology $\SH^+(W;\alpha)$ can be translated into a bound for some kind of Hofer--Zehnder capacity. For $\alpha\neq\one$, this idea has been explored by Biran, Polterovich and Salamon in \cite{BPS03}. Inspired by the work of Gatien and Lalonde \cite{GL00}, they introduced a relative capacity as follows. Let $U$ be a symplectic domain possibly with boundary, and let $Z\subset U\setminus\p U$ be a compact subset. Let $\mathcal H_\mathrm{BPS}(U,Z)$ be the set of smooth Hamiltonians $H\colon S^1\x U \to\R$ such that $H$ vanishes outside a compact subset of $U\setminus\p U$ and $\max_{S^1\x Z}H$ is negative. Then, the relative capacity with class $\alpha\in[S^1,U]$, $\alpha\neq\one$ introduced in \cite{BPS03} is given by
\[
c_\mathrm{BPS}(U,Z;\alpha):=\sup_{H\in\mathcal H_\mathrm{BPS}(U,Z)}\Big\{ -\max_{S^1\x Z}H\ \big|\ \mathcal P_1(H;\alpha)=\varnothing\Big\}\in(0,\infty],
\]
where $\mathcal P_1(H;\alpha)$ is the set of elements of $\mathcal P_1(H)$ in the class $\alpha$.
By definition, if $c_\mathrm{BPS}(U,Z;\alpha)$ is finite, we can infer the existence of a one-periodic orbit in the class $\alpha$ for every $H\in\mathcal H_\mathrm{BPS}(U,Z)$ with $-\max_{S^1\x Z}H>c_\mathrm{BPS}(U,Z;\alpha)$. Moreover, much as in the case of the Hofer--Zehnder capacity, finiteness of $c_\mathrm{BPS}$ implies the almost existence theorem for hypersurfaces $\Sigma\subset U\setminus\p U$ bounding a compact region containing $Z$ in its interior.

For the unit-disc cotangent bundle $D^*Q$ of $Q$ with Finsler metric $F$ there holds
\begin{equation}\label{e:compcot}
\ell_\alpha=c_\mathrm{BPS}(D^*Q,Q;\alpha),\qquad\alpha\neq\one
\end{equation}
where $\ell_\alpha$ is the minimal $F$-length of a closed curve in $Q$ in the class $\alpha\in[S^1,Q]\cong[S^1,D^*Q]$. For Riemannian metrics, this is established by Biran, Polterovich and Salamon in \cite{BPS03} for $Q=\T^n$ and by Weber \cite{Web06} for all closed manifolds $Q$. For general Finsler metrics, this is proved by Gong and Xue in \cite{GX19}. As a result, the almost existence theorem for non-contractible orbits holds for compact hypersurfaces which bound a compact region containing the zero-section in its interior. This strengthens the nearby existence Theorem \ref{t:HV} of Hofer and Viterbo for non-simply connected manifolds $Q$.

For Liouville domains $W$ and any class $\alpha\in[S^1,W]$, we can now use $\SH^+(W;\alpha)$ to define $c_{\SH^+}(W;\alpha)\in(0,\infty]$ as the minimal filtration level at which a non-zero class of $\SH^+(W;\alpha)$ appears, see \eqref{eq:sh+}. Thus, the finiteness of $c_{\SH^+}(W;\alpha)$ is equivalent to the non-vanishing of $\SH^+(W;\alpha)$. For contractible loops, we use the notation $c_{\SH^+}(W):=c_{\SH^+}(W;\one)$ and we have $c_{\SH^+}(W)\leq c_{\SH}(W)$ as observed in Lemma \ref{lem:c_SH=c_SH_+}. For non-contractible loops, Weber proved in \cite{Web06} (although the result is not explicitly stated) that 
\begin{equation}\label{e:bpssh}
 c_\mathrm{BPS}(W,W_\sk;\alpha)\leq c_{\SH^+}(W;\alpha),\qquad\alpha\neq\one.
\end{equation}
For cotangent bundles the filtered version of Viterbo's isomorphism \eqref{e:viterbo-iso} yields
\[
\ell_\alpha=c_{\SH^+}(D^*Q;\alpha),\qquad \alpha\neq\one,
\]
so that $\ell_\alpha$, $c_\mathrm{BPS}(D^*Q,Q;\alpha)$ and $c_{\SH^+}(D^*Q;\alpha)$ all coincide.

\section{Statement of main results}\label{s:results}

The present paper originates as an attempt to study a counterpart to the inequality \eqref{e:bpssh} for $\alpha=\one$. To this purpose, we need to replace the Biran--Polterovich--Salamon capacity with another relative Hofer--Zehnder capacity which we now define following the work of Ginzburg and G\"urel \cite{GG04}. For a symplectic manifold $U$ possibly with boundary and a compact subset $Z$ of $U\setminus\p U$, we define the set $\mathcal H(U,Z)$ of smooth Hamiltonians $H\colon U\to\R$ such that $H$ vanishes outside a compact subset of $U\setminus\p U$ and $H=\min H<0$ on a neighborhood of $Z$. Then the relative Hofer--Zehnder capacity is given by 
\[
c_{\HZ}(U,Z):=\sup_{H\in\mathcal H(U,Z)}\big\{ -\min H\ \big|\ \mathcal P_{\leq 1}(H)=\mathrm{Crit}(H)\big\} \in (0,\infty].
\]
Considering only orbits in the class $\alpha\in[S^1,U]$, we can also define $c_{\HZ}(U,Z;\alpha)$. By definition, 
\begin{equation}\label{eq:HZ_BPS}
c_{\HZ}(U,Z;\alpha)\leq c_\mathrm{BPS}(U,Z;\alpha),\qquad \alpha\neq\one.	
\end{equation}
Moreover for Liouville domains $W$, one can easily see 
\begin{equation}\label{eq:min_HZ}
\min\mathrm{spec}(R_\lambda;\alpha)\leq c_{\HZ}(W,W_\sk;\alpha)\qquad \forall \alpha\in [S^1,W]	
\end{equation}
where $\mathrm{spec}(R_\lambda;\alpha)$ is the set of periods of Reeb orbits of $(\p W,\lambda|_{\p W})$ with the class $\alpha$ in $W$.

From now on, we focus on contractible orbits and write $c_{\HZ}^o(U,Z):=c_{\HZ}(U,Z;\one)$. In order to deal with time-dependent Hamiltonians as well, we introduce a slightly different version of $c_{\HZ}^o(U,Z)$ when the symplectic form $\omega=\di\lambda$ is exact on $U$. We consider the set $\widetilde{\mathcal H}(U,Z)$ consisting of smooth time-dependent Hamiltonians $H\colon S^1\x U\to\R$ such that $H$ vanishes outside a compact subset of $S^1\times(U\setminus \p U)$ and $H=\min H<0$ on a neighborhood of $S^1\x Z$. The action of $p$-periodic loops $\gamma\colon\R/p\Z\to  U$ with respect to $H\colon S^1\x U \to\R$ is given by
\begin{equation}\label{e:action}
\mathcal A_H(\gamma)=\int_0^p\gamma^*\lambda-\int_0^p H(t,\gamma(t))\,\di t	.
\end{equation}
For all $a\in(0,\infty]$, we define
\[
\widetilde{c}_\HZ^{\,o}(U,Z,a):=\!\sup_{H\in\widetilde{\mathcal H}(U,Z)}\!\big\{-\min H\ \big|\ \forall x\in \mathcal P_1^o(H),\;\; \mathcal A_H(x)\notin (-\min H,-\min H+a]\big\}\in(0,\infty].
\]
By definition, there holds
\begin{equation}\label{e:cc}
c_{\HZ}^{\,o}(U,Z)\leq \widetilde{c}_{\HZ}^{\,o}(U,Z,a)\qquad \forall a\in(0,\infty].
\end{equation}
Moreover, every $H\in \mathcal {\widetilde H}(U,Z)$ with $-\min H>\widetilde{c}_{\HZ}^{\,o}(U,Z,a)$ has a (non-constant) contractible one-periodic orbit $x$ with
\begin{equation}\label{e:aa}
\mathcal A_H(x)\in(-\min H,-\min H+a].
\end{equation}
Building on these two facts, we obtain the following implications of this newly defined capacity to periodic orbits in the autonomous and non-autonomous setting. To this end, recall that a free-homotopy class $\alpha\in[S^1,U]$ is called torsion if $\alpha^p=\one$ for some $p\in\N$.

\begin{prp}\label{p:finite}
Let $(U,\di\lambda)$ be an exact symplectic manifold possibly with boundary and let $Z$ be a compact subset of $U\setminus \partial U$. 
\begin{enumerate}[(a)]
\item If $\widetilde{c}_{\HZ}^{\,o}(U,Z,a)<\infty$ for some $a\in(0,\infty]$, then the almost existence theorem for contractible orbits holds for every hypersurface in $U\setminus \p U$ bounding a compact region containing $Z$ in its interior.
\item Assume that $\widetilde{c}_{\HZ}^{\,o}(U,Z,a)<\infty$ for some $a<\infty$. If $H\in\widetilde{\mathcal H}(U,Z)$ has only finitely many one-periodic orbits with torsion free-homotopy classes and with action greater than $-\min H$, then for every sufficiently large prime number $p$, there exists a contractible, non-iterated, non-constant, $p$-periodic orbit of $H$  which has action greater than $-p\min H$. In particular every $H\in\widetilde{\mathcal H}(U,Z)$ carries infinitely many distinct contractible, non-constant periodic orbits.
\end{enumerate}
\end{prp}
\begin{proof}
Suppose that $\widetilde{c}_{\HZ}^{\,o}(U,Z,a)$ is finite for some $a\in(0,\infty]$. Then by \eqref{e:cc}, $c^o_\HZ(U,Z)$ is also finite, and the almost existence theorem for contractible orbits follows from \cite[Theorem 2.14]{GG04}. This proves (a). 

To show (b), we assume that $a$ is finite. Let $H\in\widetilde\HH(U,Z)$ be as in the statement, and let $\gamma_1,\dots,\gamma_m$ be all one-periodic orbits of $H$ with torsion free-homotopy classes and with action greater than $-\min H$. We choose $\epsilon>0$ such that 
\[
\AA_H(\gamma_i)\geq -\min H+\epsilon \qquad \forall i\in\{1,\dots, m\}.
\] 
For $p\in\N$, we define the $p$-th iteration $H^{\natural p}\colon S^1\x W\to\R$ of $H$ by $H^{\natural p}(t,z):= pH(pt,z)$ so that the Hamiltonian flow of $H^{\natural p}$ and that of $H$ are related by $\phi_{H^{\natural p}}^t=\phi_H^{pt}$. Thus one-periodic orbits of $H^{\natural p}$ can be viewed as $p$-periodic orbits of $H$. If $\gamma\colon S^1\to W$ is a one-periodic orbit of $H$, then its $p$-th iteration $\gamma^p\colon S^1\to W$ defined by $\gamma^p(t):=\gamma(pt)$ is a one-periodic orbit of $H^{\natural p}$. Moreover there obviously holds $\AA_{H^{\natural p}}(\gamma^p)=p\AA_H(\gamma)$. Let now $p$ be a prime number so large that 
\[
\epsilon p> a,\qquad -\min H^{\natural p}=-p\min H>\widetilde c_{\HZ}^o(U,Z,a).
\] 
By the definition of $\tilde c_{\HZ}^o(U,Z,a)$, $H^{\natural p}$ has a contractible one-periodic orbit $\gamma_\mathrm{new}$ such that 
\[
-\min H^{\natural p}<\AA_{H^{\natural p}}(\gamma_\mathrm{new})\leq-\min H^{\natural p}+a.
\] 
The first inequality shows that $\gamma_\mathrm{new}$ is non-constant. The latter one yields that $\gamma_\mathrm{new}$ is non-iterated. Indeed if $\gamma_\mathrm{new}$ is iterated, then it has to be the $p$-th iteration of $\gamma_i$ for some $1\leq i\leq m$ and this is absurd since
\[
\AA_{H^{\natural p}}(\gamma_i^p)=p\AA_{H}(\gamma_i)\geq p(-\min H+\epsilon)>-\min H^{\natural p}+a,\qquad \forall i=1,\ldots,m.
\]
This finishes the proof of (b).
\end{proof}

We are now ready to state our main result.
It says that the $\pi_1$-sensitive Hofer--Zehnder capacity of a Liouville domain relative to its skeleton can be bounded by the capacity obtained from positive symplectic homology in the contractible class.

\begin{thm}\label{main_thm}
For every Liouville domain $W$, there holds
\[
\tilde{c}_\HZ^{\,o}(W,W_\sk,a)\leq c_{\SH^+}(W) \qquad \forall a\in\big[c_{\SH^+}(W),\infty\big].
\]
Hence if $c_{\SH^+}(W)$ is finite, or equivalently $\SH^+(W;\one)$ is non-zero, then the same conclusion as in Proposition \ref{p:finite} holds for $(U,Z)=(W,W_\sk)$.
\end{thm}

\begin{rmk}
The hypothesis  $\SH^+(W;\one)\neq 0$ in Theorem \ref{main_thm} is indispensable. For example, if $W$ is a fiberwise starshaped domain in $T^*S^1$, then $\SH^+(W;\one)=0$ and none of (a) and (b) in Proposition \ref{p:finite} is true.
\end{rmk}

\begin{rmk}
For $\delta>0$, let
\begin{equation}\label{e:M_delta}
W^{\delta}:=\phi_Y^{\log \delta}(W).	
\end{equation}
Then the proof of Theorem \ref{main_thm} actually shows that for any $\delta\in(0,1]$, 
\[
\widetilde{c}_\HZ^{\,o}(W,W^{\delta},a)\leq (1-\delta) c_{\SH^+}(W)\qquad  \forall a\in \big[(1-\delta) c_{\SH^+}(W),\infty\big]
\]
and this subsumes Theorem \ref{main_thm} since $W_\sk=\bigcap_{\delta>0}W^\delta$.
\end{rmk}

Theorem \ref{main_thm} will follow from Proposition \ref{prp:lower_bound}, which provides a lower bound on the number of contractible one-periodic orbits of $H\in\widetilde\HH(W,W_\sk)$ with action in a certain interval in terms of positive symplectic homology. Combining Theorem \ref{main_thm} with the isomorphism \eqref{e:viterbo-iso}, we immediately obtain the following corollary in $T^*Q$.

\begin{cor}\label{c:cot}
Let $Q$ be a closed manifold such that $\H(\mathcal L_\one Q,Q)$ is non-zero. 
\begin{enumerate}[(a)]
	\item The almost existence theorem for contractible orbits holds for every hypersurface of $T^*Q$ bounding a compact region containing the zero-section.
	\item Every compactly supported smooth Hamiltonian $H\colon S^1\x T^*Q\to \R$ with  $H=\min H<0$ on $S^1\x U$, where $U$ is a neighborhood of the zero-section, has infinitely many distinct non-constant, contractible periodic orbits with action greater than $-\min H$. \hfill\qed
\end{enumerate}
\end{cor}
\begin{rmk}\label{r:rot}
Let $Q$ be simply connected. By the theory of minimal models of Sullivan \cite{Sul75,VPS76}, the group $\H(\mathcal L_\one Q,Q)$ is infinite dimensional and, in particular, non-zero. As observed by Thomas Rot in a MathOverflow post \cite{Rot}, using the (relative) Hurewicz theorem and the long exact sequence of the pair $(\mathcal L_\one Q,Q)$ in homology and homotopy, one can show that $\H_{k-1}(\mathcal L_\one Q,Q)\neq0$, if $k$ is the smallest integer such that $\pi_k(Q)\neq0$. When $\pi_1(Q)\cong\Z$ conditions on the homotopy groups of $Q$ ensuring $\H(\mathcal L_\one Q,Q)\neq0$ are given in \cite[Corollary 1.9]{AGKM}.
\end{rmk}

\begin{rmk}
The corollary finds application also to exact twisted cotangent bundles. Let $\di\theta$ be an exact two-form on a closed manifold $Q$ and consider the twisted cotangent bundle $(T^*Q,\omega_{T^*Q}+\pi^*(\di\theta))$, where $\omega_{T^*Q}$ is the canonical symplectic form on $T^*Q$ and $\pi$ is the foot-point projection $\pi\colon T^*Q\to Q$. If $\H(\mathcal L_\one Q,Q)\neq0$, then the statements in Corollary \ref{c:cot} hold by replacing the zero-section with the graph of the one-form $\theta$.
\end{rmk}
If $D^*Q$ is the unit disc cotangent bundle of a Riemannian metric $g$, the exact value of $c_{\SH^+}(D^*Q)$ and of $\tilde c_{\HZ}^o(D^*Q,Q,a)$ for $a\geq c_{\SH^+}(D^*Q)$ can be computed via Viterbo isomorphism if we know the homology of $\mathcal L_\one Q$ filtered by the square root of the $g$-energy of loops sufficiently well. For instance, applying a theorem of Ziller \cite{Zil} and Lemma \ref{l:bir=sys} contained in the appendix of the present paper written by Abbondandolo and Mazzucchelli, we get the following statement.
\begin{cor}\label{c:cross}

Let $Q$ be a closed manifold endowed with a Riemannian metric $g$. Let $D^*Q$ be the unit-disc cotangent bundle of $g$ and denote by $\ell_\one$ the length of the shortest non-constant contractible geodesic for the metric. If $(Q,g)$ is a compact, non-aspherical homogeneous space (for instance a compact rank one symmetric space) or $(Q,g)$ is a two-sphere with positive curvature, then there holds
\[
\ell_\one=\tilde c_{\HZ}^o(D^*Q,Q,a)=c_{\SH^+}(D^*Q), \qquad \forall a\in\big[c_{\SH^+}(D^*Q),\infty\big].
\]
\end{cor}

Finally, we can adapt \cite[Theorem 3.2]{GG04} to obtain weaker statements on the existence of periodic orbits for $H$ belonging to a class larger than $\widetilde{\mathcal H}(W,W_\sk)$ of functions that are allowed to be time-dependent also on $W_\sk$. The definition of the Floer homology $\HF$ and the canonical map $\iota_{-\infty}^{\epsilon,\infty}$ is given in Section \ref{sec:Floer_homology}.

\begin{thm}\label{thm:nonautonomous}
Let $(W,\lambda)$ be a Liouville domain, and let $H\colon S^1\x W\to\R$ be a Hamiltonian which is supported in $S^1\times(W\setminus\p W)$ and satisfies $\max_{S^1\x W_\sk} H<0$.
\begin{enumerate}[(a)]
		\item For every small $a>0$, there holds 
	\[
	\rk \HF^{(a,\infty)}(H) \geq \rk\left[\iota_{-\infty}^{\epsilon,\infty}\colon\H(W,\p W)\to \SH(W;\one)\right].
	\]
	In particular, if $\SH(W;\one)$ is nonzero, then $H$ has a contractible one-periodic orbit with positive action. 
		\item Assume in addition that $ \max_{S^1\x W} H=0$. For every small $a>0$, there exists a surjective homomorphism
	\[
	\HF^{(a,\infty)}(H)\longrightarrow \H(W,\p W).
	\]
	In particular, $H$ has a contractible one-periodic orbit with positive action.
	\end{enumerate}
\end{thm}
\subsection*{Organization of the paper}
In Section \ref{sec:Floer_homology}, we recall the precise definition of Hamiltonian Floer homology, of symplectic homology, of the associated capacities. At the end of the section, a proof of Theorem \ref{t:cgamma} is given. In Section \ref{sec:proofs} we prove Theorem \ref{main_thm},  Theorem \ref{thm:nonautonomous}, and Corollary \ref{c:cross}. Appendix \ref{appendix}, written by Abbondandolo and Mazzucchelli, shows the monotonicity of the systoles for convex Riemannian two-spheres in $\R^3$. In doing so, they prove Lemma \ref{l:bir=sys} which is needed in Corollary \ref{c:cross}.
\subsection*{Acknowledgements} G.~Benedetti acknowledges funding by the Deutsche Forschungsgemeinschaft (DFG, German Research Foundation) through the projects: 281869850 (RTG 2229); 390900948 (EXC-2181/1); 281071066 (SFB/TRR 191). J.~Kang was supported by Research Resettlement Fund for the new faculty of Seoul National University and by TJ Park Science Fellowship of POSCO TJ Park Foundation. We are grateful to Thomas Rot for discussions around Remark \ref{r:rot}, to Pierre-Alexandre Mailhot and Egor Shelukhin for pointing out a mistake in a previous version of the manuscript, and to Viktor Ginzburg and Kei Irie for their feedback on the paper. 

\section{Floer homologies and symplectic capacities }\label{sec:Floer_homology}
In this section we define the capacities given by symplectic homology and prove Theorem \ref{t:cgamma}. Prior to this we give a concise construction of Floer homology and refer to \cite{CFH95,Sal99,Vit99,BPS03,GG04,Web06} for details.

\subsection{Hamiltonian Floer homology}\label{sec:Hamiltonian}
For a Liouville domain $(W,\lambda)$, let $(\widehat W,\hat\lambda)$ be its completion defined in \eqref{eq:completion} and $W^\delta\subset \widehat W$ for $\delta>0$ be the subset given in \eqref{e:M_delta}. We consider a smooth Hamiltonian $H\colon S^1\x \widehat W\to\R$ with
\begin{equation}\label{eq:H_infty}
H(t,r,x)= \tau r + \eta \qquad (t,r,x)\in  S^1\x (\widehat W\setminus W^\delta)=S^1\x (\delta,\infty)\x\p W
\end{equation}
for some $\delta>0$, $\tau\in(0,\infty)\setminus\mathrm{spec}(R_\lambda)$ and $\eta\in\R$. The constant $\tau$ is called the slope of $H$. The action spectrum $\mathrm{spec}(H)$ is the set of action values of all critical points of $\AA_H$. This is a compact, nowhere dense subset of $\R$. Let $a\leq b$ be numbers in $\overline{\R}:=\R\cup\{-\infty,\infty\}$ not belonging to $\mathrm{spec}(H)$. We denote by $\PP_1^{(a,b)}(H;\alpha)$ the set of one-periodic orbits of $H$ with free-homotopy class $\alpha$ and with action in $(a,b)$. Suppose that all elements in $\PP_1^{(a,b)}(H;\alpha)$ are nondegenerate. Then this is a finite set due to $\tau\notin\mathrm{spec}(R_\lambda)$.   The Floer chain group is
\[
\mathrm{CF}^{(a,b)}(H;\alpha):=\bigoplus_{\substack{x\in\PP_1^{(a,b)}(H;\alpha)}}  \Z_2\langle x\rangle .
\]
Let $J$ be a smooth $S^1$-family of almost complex structures on $\widehat W$ with the property that $\diff\hat\lambda(\cdot,J(t,u)\cdot)$ is an inner product on $T_u\widehat W$ for all $(t,u)\in S^1\times \widehat W$ and satisfying 
\begin{equation}\label{e:J}
J^*\hat\lambda=\di r, \qquad \text{on }\{r\geq\delta_0\}
\end{equation}
for some $\delta_0>0$ such that all element in $\PP_1^{(a,b)}(H;\alpha)$ are contained in the interior of $W^{\delta_0}$. 
For $x,y\in\PP_1(H;\alpha)$, we denote by $\mathcal{M}(x,y)$ the moduli space of Floer cylinders connecting $x$ and $y$, namely smooth solutions $u\colon\R\x S^1\to \widehat W$ of 
\begin{equation}\label{eq:Floer}
\p_su-J(t,u)\big(\p_tu-X_H(t,u)\big)=0,\qquad \lim_{s\to -\infty}u(s,\cdot)=x,\qquad \lim_{s\to +\infty}u(s,\cdot)=y,
\end{equation}
where $(s,t)\in\R\x S^1$ and $X_H$ is the Hamiltonian vector field of $H$ defined via the equation $\diff\hat\lambda(X_H,\cdot)=-\diff H$. Unless $x=y$, there is a free $\R$-action on $\mathcal{M}(x,y)$ by translating solutions in the $s$-direction. We define $n(x,y)$ as the parity of $\mathcal{M}(x,y)/\R$ if $x\neq y$ and it is a finite set. Otherwise we set $n(x,y)=0$. The differential $\p\colon\mathrm{CF}^{(a,b)}(H;\alpha)\to \mathrm{CF}^{(a,b)}(H;\alpha)$ is defined by the linear extension of the formula
\[
\qquad \p x := \sum_{\substack{{y\in \PP_1^{(a,b)}(H;\alpha)}}} n(x,y)y.
\]
For a generic choice of $J$, we indeed have $\p\circ \p=0$ and denote the Floer homology of $H$ with action-window $(a,b)\subset{\R}$ and with free homotopy class $\alpha\in[S^1,\widehat W]$ by
\[
\HF^{(a,b)}(H;\alpha):=\H(\mathrm{CF}^{(a,b)}(H;\alpha),\p).
\]
Simplifying the notation we denote $\HF(H;\alpha)=\HF^{(-\infty,\infty)}(H;\alpha)$.
As the notation indicates, a different choice of $J$ produces an isomorphic homology via a continuation homomorphism.

Given $a<b<c$ not belonging to $\mathrm{spec}(H)$, the exact sequence of chain complexes
\[
0 \longrightarrow \CF^{(a,b)}(H;\alpha)\longrightarrow \CF^{(a,c)}(H;\alpha) \longrightarrow  \CF^{(b,c)}(H;\alpha) \longrightarrow 0
\]
induced by natural inclusion and projection gives rise to the long exact sequence 
 \begin{equation}\label{eq:les1}
\cdots\stackrel{\delta}{\to} \HF^{(a,b)}(H;\alpha)\stackrel{\iota}{\to} \HF^{(a,c)}(H;\alpha) \stackrel{\pi}{\to} \HF^{(b,c)}(H;\alpha) \stackrel{\delta}{\to} \HF^{(a,b)}(H;\alpha) \stackrel{\iota}{\to} \cdots.
\end{equation}
Let $H,K\colon S^1\x\widehat W\to\R$ be two smooth Hamiltonians with the aforementioned properties and $H\leq K$. We choose a smooth monotone homotopy $H_s$, $s\in\R$ from $H$ to $K$, namely
\[
H_s=H \quad  \forall s\leq -1,\qquad H_s=K\quad \forall s\geq 1,\qquad \p_s H_s\geq0\quad \forall s\in\R,
\]  
and $H_s$ has a constant slope, see \eqref{eq:H_infty}, for every $s\in\R$.
Consider the moduli space of solutions of \eqref{eq:Floer} with $H$ replaced by $H_s$ and define the continuation homomorphism 
\begin{equation}\label{eq:continuation}
\Phi=\Phi_{H,K}^{(a,b)}\colon\HF^{(a,b)}(H;\alpha)\longrightarrow \HF^{(a,b)}(K;\alpha)	
\end{equation}
in an analogous way to defining the differential. Another choice of monotone homotopy produces the same continuation homomorphism. Moreover the map $\Phi$ induces a commuting map from the exact sequence \eqref{eq:les1} for $H$ to that for $K$.
If we consider another smooth Hamiltonian $G\colon S^1\x\widehat W\to\R$ satisfying $K\leq G$, then we have continuation homomorphisms $\Phi_{K,G}^{(a,b)}$ and $\Phi_{H,G}^{(a,b)}$ and there holds $\Phi_{H,G}^{(a,b)}=\Phi_{K,G}^{(a,b)}\circ\Phi_{H,K}^{(a,b)}$. In the case that $H$ and $K$ have the same slope and $(a,b)=(-\infty,\infty)$, continuation homomorphisms $\Phi_{H,K}=\Phi_{H,K}^{(-\infty,\infty)}$ and $\Phi_{K,H}$ are still defined and satisfy $\Phi_{H,K}\circ \Phi_{K,H}=\Phi_{K,K}=\mathrm{Id}$ and $\Phi_{K,H}\circ \Phi_{H,K}=\Phi_{H,H}=\mathrm{Id}$. Thus $\HF(H;\alpha)\cong\HF(K;\alpha)$ for $H$ and $K$ with the same slope.

In fact the above construction extends to smooth Hamiltonians $H\colon S^1\x \widehat W\to\R$ such that elements $\PP_1^{(a,b)}(H;\alpha)$ are not necessarily nondegenerate. The nondegeneracy condition can be achieved by a small compact perturbation $K$ of $H$. Moreover the Floer homology $\HF^{(a,b)}(K;\alpha)$ is independent of the choice of a small perturbation up to continuation isomorphisms. To be precise, if $G$ is another small compact perturbation of $H$, then we have continuation homomorphisms $\Phi^{(a,b)}_{K,G}$ and $\Phi^{(a,b)}_{G,K}$ which are inverse to each other. Here it is crucial that $a,b\notin\mathrm{spec}(H)$, and $K$ and $G$ have the same slope.
Thus we set $\HF^{(a,b)}(H;\alpha):=\HF^{(a,b)}(K;\alpha)$.

Next we define the Floer homology of $H\colon S^1\x W\to\R$ with support in $S^1\times (W\setminus\p W)$. We choose $\delta_1\in(0,1)$ such that $H=0$ on $W\setminus W^{\delta_1}$. Then we smoothly extend $H$ to $\widehat  H\colon S^1\x\widehat W\to \R$ to satisfy
\begin{itemize}
	\item $\widehat H=H$ on $W\setminus W^{\delta_2}$ for some $\delta_2\in(\delta_1,1)$;
	\item $\widehat H(t,r,x)=h(r)$ on $(t,r,x)\in\widehat W\setminus W^{\delta_1}=S^1\x (\delta_1,\infty)\x\p W$ where $h\colon(\delta_1,\infty)\to\R$ is a smooth convex function;
	\item $h'(r)=\epsilon$ for some $0<\epsilon<\min\mathrm{spec}(R_\lambda)$ on $\widehat W\setminus W$.
\end{itemize}
Then we define the Floer homology of $H$ as that of $\widehat H$:
\begin{equation}\label{eq:fh_cpt_supp}
\HF^{(a,b)}(H;\alpha):=\HF^{(a,b)}(\widehat H;\alpha)	
\end{equation}
where $a,b\in\overline{\R}\setminus\mathrm{spec}(H)$ as usual. Due to the choice of slope $\epsilon$, $\PP_1(H)=\PP_1(\widehat H)$ and furthermore the definition \eqref{eq:fh_cpt_supp} is independent of the choice of $\widehat H$.

Finally, we make an action computation that will be repeatedly used. 
Let $H\colon S^1\x W\to\R$ be such that there exist $\delta>0$ and a smooth function $h\colon(\delta,\infty)\to\R$ with the property that $H=h$ on $S^1\x (W\setminus W^\delta)=S^1\x (\delta,\infty)\x \p W$. In this case, the action of one-periodic orbits $x$ of $H$ located on $\p W^r$ for $r\in (\delta,\infty)$ is explicitly computed as 
\begin{equation}\label{e:action_computation}
	\AA_H(x)= rh'(r)-h(r),
\end{equation} 
which is minus the $y$-intercept of the tangent line of $h$ at $r$.
\subsection{Symplectic homology}
Let $\HH^{a,b}$ for $a\leq b$ in $\overline\R$ be the set of smooth Hamiltonians $H\colon S^1\x\widehat W\to\R$ satisfying $a,b\notin\mathrm{spec}(H)$, $H|_{S^1\x W}<0$, and \eqref{eq:H_infty} for some $\delta\geq1$, $\tau\in(0,\infty)\setminus\mathrm{spec}(R_\lambda)$, and $\eta\in\R$.  We endow $\HH^{a,b}$ with the partial relation $\leq$ given by the pointwise inequality so that for every $H,K\in \HH^{a,b}$ with $H\leq K$, we have the continuation homomorphism defined in \eqref{eq:continuation}. Floer homology groups of elements in $\HH^{a,b}$ together with continuation homomorphisms form a direct system, and the direct limit is called the symplectic homology of $W$:
\[
\SH^{(a,b)}(W;\alpha):=\varinjlim_{H\in\HH^{a,b}}\HF^{(a,b)}(H;\alpha).
\]
We remark that the symplectic homology changes only when the action-window crosses $\mathrm{spec}(R_\lambda)$, i.e.~$\SH^{(a,b)}(W;\alpha)\cong\SH^{(a',b')}(W;\alpha)$ if $(a,b)\cap\mathrm{spec}(R_\lambda)=(a',b')\cap \mathrm{spec}(R_\lambda)$. Let $\epsilon$ denote a constant such that $0<\epsilon<\min\mathrm{spec}(R_\lambda)$. Thus $(-\infty,\epsilon)\cap \mathrm{spec}(R_\lambda)=\varnothing$, and there holds
\begin{equation}\label{eq:SH_epsilon}
\SH^{(-\infty,\epsilon)}(W;\alpha)\cong 
\left\{ \begin{aligned} & \H(W,\p W) \;\;&\alpha=\one\,, \\[0.5ex]
 	& \; 0 & \alpha\neq\one\,,
\end{aligned}\right.
\end{equation}
where $\H(W,\p W)$ is the relative homology of the pair $(W,\p W)$. We denote
\[
\SH(W;\alpha):=\SH^{(-\infty,+\infty)}(W;\alpha),\qquad \SH^+(W;\alpha)=\SH^{(\epsilon,\infty)}(W;\alpha).
\]
For any $a\leq b\leq c$ in $\overline{\R}\setminus\mathrm{spec}(R_\lambda)$, the exact sequence \eqref{eq:les1} leads to the exact sequence
\begin{equation}\label{eq:les_SH}
\cdots\stackrel{\delta}{\to} \SH^{(a,b)}(W;\alpha)\stackrel{\iota}{\to} \SH^{(a,c)}(W;\alpha) \stackrel{\pi}{\to} \SH^{(b,c)}(W;\alpha) \stackrel{\delta}{\to} \SH^{(a,b)}(W;\alpha)\stackrel{\iota}{\to} \cdots.
\end{equation}
We decorate $\iota$ to indicate involved action-windows as follows:
\begin{equation}\label{eq:iota}
\iota_a^{b,c}:\SH^{(a,b)}(W;\alpha)\longrightarrow \SH^{(a,c)}(W;\alpha).	
\end{equation}
This map is functorial in the sense that $\iota_a^{c,d}\circ\iota_a^{b,c}=\iota_a^{b,d}$ holds for any $d\geq c$. Indeed, the map $\iota$ in \eqref{eq:les1} defined for each $H$ has such a property and is compatible with the continuation homomorphism in \eqref{eq:continuation}. Thus the desired functorial property for  symplectic homology follows.
Applying the exact sequence \eqref{eq:les_SH} to $(a,b,c)=(-\infty,\epsilon,\infty)$, we deduce
\[
\SH^+(W;\alpha)\cong\SH(W;\alpha),\qquad \alpha\neq\one
\]
and 
\begin{align*}
\dim \SH(W;\one)=\infty \qquad&\Longleftrightarrow\qquad \dim \SH^+(W;\one)=\infty,\\
\SH(W;\one)=0\qquad &\Longrightarrow\qquad \SH^+(W;\one)\cong \H(W;\p W).
\end{align*}

Symplectic homologies with finite action-window and the homomorphisms in \eqref{eq:iota} can be interpreted as Floer homologies of suitably chosen Hamiltonians and continuation homomorphisms between them. For $a\in(0,\infty)\setminus\mathrm{spec}(R_\lambda)$ we consider the set $\mathcal G_a$ of smooth functions $g_a\colon\widehat W\to \R$ such that there are positive numbers $\epsilon',\delta,c$ with $\delta<1$ depending on $g_a$ with
\begin{itemize}
	\item $g_a=-\epsilon'$ on $W^{1-\delta}$;
	\item on $\widehat W\setminus W^{1-\delta}$, the function $g_a$ depends only on $r$ and there holds $g''_a(r)\geq0$;
	\item $g_a=a(r-1)-c$ on $\widehat W\setminus W$.
\end{itemize}
A non-constant one-periodic orbit of $g_a$ sits in $\p W^r$ for $r>1-\delta$ such that $g'(r)\in\mathrm{spec}(R_\lambda)$, and corresponds to a closed Reeb orbit of $(\p W,\lambda|_{\p W})$ with period $g'(r)$. 
If we consider the piecewise linear function
\[
\bar g_a\colon\widehat W\to\R,\qquad \bar g_a|_W=0,\quad \bar g_a|_{\widehat W\setminus W}=a(r-1),
\] 
then choosing $\epsilon',\delta,c$ small enough, the function $g_a$ can be arbitrarily $C^0$-close to $\bar g_a$ on $\widehat W$ and $C^\infty$-close to $\bar g_a$ away from $\p W$, and furthermore the action of a non-constant one-periodic orbit of $g_a$ can be arbitrarily close to the period of the corresponding Reeb orbit by \eqref{e:action_computation}.

\begin{lem}\label{lem:isom}
Let $\epsilon,a,b$ be real numbers such that 
\[
0<\epsilon<\min\mathrm{spec}(R_\lambda)<a<b,\qquad  a,b\notin\mathrm{spec}(R_\lambda).
\] 
There exist $g_a\in\mathcal G_a$ and $g_b\in\mathcal G_b$ which can be taken arbitrarily $C^0$-close to $\bar g_a$ and $\bar g_b$ respectively such that the following diagram commutes
\begin{equation}\label{e:comlem}
\begin{gathered}
	\xymatrix{
\HF^{(\epsilon,b)}(g_b;\alpha) \ar^{\cong}_{\phi}[r]  & \SH^{(\epsilon,b)}(W;\alpha)\\
\HF^{(\epsilon,b)}(g_a;\alpha) \ar[u]_{\Phi} \ar@/^-1.5pc/@[][ru]_-{\phi}    &
	\\
\HF^{(\epsilon,a)}(g_a;\alpha) \ar[u]_{\iota}^{\cong} \ar[r]_{\phi}^{\cong} &  \SH^{(\epsilon,a)}(W;\alpha).\ar_{\iota_{\epsilon}^{a,b}}[uu]
	}
\end{gathered}
	\end{equation}
Here $\Phi$ is a continuation homomorphism, $\iota$ is a homomorphism from \eqref{eq:les1}, $\iota^{a,b}_\epsilon$ is from \eqref{eq:iota}, and the maps $\phi$ are homomorphisms in the direct system. 
\end{lem}
\begin{proof}
For any increasing sequence $(a_i)_{i\in\N}\subset(0,\infty)\setminus\mathrm{spec}(R_\lambda)$ with $a_1=a$, we choose a sequence of functions $g_{a_i}\in\mathcal G_{a_i}$ such that $(g_{a_i})$ is cofinal in $\HH^{\epsilon,a}$ with $g_{i+1}\geq g_i$ and for each $i\in\N$ there is a monotone homotopy $(g^s)_{s\in[0,1]}$ from $g_{a_i}$ to $g_{a_{i+1}}$ with the property that for all $s\in[0,1]$, the function $g^s$ has no one-periodic orbit $x$ with $\mathcal A_{g^s}(x)\in\{\epsilon,a\}$. Then the continuation homomorphism induced by $(g^s)$, 
\[
\Phi\colon\HF^{(\epsilon,a)}(g_{a_i};\alpha) \stackrel{\cong}{\longrightarrow}   \HF^{(\epsilon,a)}(g_{a_{i+1}};\alpha)
\]
	is an isomorphism, see for instance \cite[Lemma 2.8]{Web06}, and thus the lower horizontal arrow in \eqref{e:comlem} is an isomorphism. The same argument shows that the upper horizontal map is also an isomorphism.
	Moreover, the map $\iota$ is an isomorphism since $g_{a_1}=g_a$ does not have one-periodic orbits with action greater than $a$. Finally, the commutativity of the diagram follows from the definitions of the involved homomorphisms.
\end{proof}
\begin{rmk}\label{r:isom}
The statement of Lemma \ref{lem:isom} holds \textit{mutatis mutandis} with $\epsilon$ replaced by $-\infty$ and $a<b$ any pair of positive numbers not in $\mathrm{spec}(R_\lambda)$.
\end{rmk}
\subsection{Capacities from Floer and symplectic homology}\label{sec:SH_capacity}
We now define the capacities mentioned in Section \ref{sec:intro}. To define the spectral invariant, we take a function $f\colon\widehat W\to\R$ in $\mathcal G_\epsilon$ where $0<\epsilon<\min\mathrm{spec}(R_\lambda)$. As discussed above, we have 
\[
\HF(f;\one)\cong \SH^{(-\infty,\epsilon)}(W;\one) \cong \H(W,\p W).
\]
We denote by $e_f\in \HF(f;\one)$ the homology class corresponding to the fundamental class in $\H(W,\p W)$ through the above isomorphism. Let $H\colon S^1\times W\to\R$ be a smooth Hamiltonian  supported in $S^1\times( W\setminus\p W)$ whose Floer homology is defined as in \eqref{eq:fh_cpt_supp}. For $a\in\R\setminus \mathrm{spec}(H)$, we consider the chain of homomorphisms
\begin{equation}\label{e:fa}
\HF(f;\one)\xrightarrow{\Phi_{f,H}}\HF(H;\one)\stackrel{\pi_a}{\longrightarrow}\HF^{(a,\infty)}(H;\one),
\end{equation}
where $\Phi_{f,H}$ is a continuation homomorphism, which is in fact an isomorphism since $f$ and $H$ have the same slope. The map $\pi_a$ is from \eqref{eq:les1}. The spectral invariant of $H$ is defined by
\[
\sigma(H):=\inf\{a\ |\ \pi_a\circ\Phi_{f,H}(e_f)=0\}.
\]
The spectral capacity $c_\sigma(W)$ of $W$ is defined by the supremum of $\sigma(H)$ over all smooth Hamiltonians $H\colon S^1\x W\to\R $ supported in $S^1\x(W\setminus\p W)$.
\medskip

 Due to \eqref{eq:SH_epsilon}, we can view the homomorphism $\iota^{\epsilon,c}_{-\infty}$ defined in \eqref{eq:iota} as a map
\[
\iota^{\epsilon,c}_{-\infty}\colon\H(W,\p W)\longrightarrow \SH^{(-\infty,c)}(W;\one).
\]
We consider the number
\begin{equation}\label{eq:c_SH}
c_{\SH}(W):=\inf\big\{c>0\ \big|\  \iota^{\epsilon,c}_{-\infty}=0\big\}\in(0,\infty].
\end{equation}
We note that due to functoriality $\iota_{-\infty}^{\epsilon,c}=0$ for any $c>c_{\SH}(W)$.
It is known that $\SH(W;\one)$ admits a ring structure with unit given by the image of the fundamental class of $\H(W,\p W)$ under $\iota_{-\infty}^{\epsilon,\infty}$, thus the quantity $c_{\SH}(W)$ is finite if and only if $\SH(W;\one)$ vanishes. Using positive symplectic homology we can also define the quantity 
\begin{equation}\label{eq:sh+}
c_{\SH^+}(W;\alpha):=\inf\{c>0\ |\ \iota_{\epsilon}^{c,\infty}\colon\SH^{(\epsilon,c)}(W;\alpha)\to \SH^+(W;\alpha)\text{ is nonzero}\}\in(0,\infty].
\end{equation}
It is finite if and only if $\SH^+(W;\alpha)\neq0$. This is equivalent to $\SH(W;\alpha)\neq0$ for $\alpha\neq \one$ and to $\SH(W;\one)\not\cong \H(W,\p W)$ for $\alpha=\one$ due to \eqref{eq:les_SH}. We use the notation $c_{\SH^+}(W)=c_{\SH^+}(W;\one)$. 

\begin{lem}\label{lem:c_SH=c_SH_+}
There holds $c_{\SH}(W)\geq c_{\SH^+}(W)$. Moreover the equality holds if $\SH(W;\one)=0$ and $\rk\H(W,\p W)=1$. 
\end{lem}
\begin{proof}
This is an immediate consequence of the commutative diagram induced by \eqref{eq:les_SH}:
\[
\xymatrix{
\cdots \ar[r]& \SH^{(\epsilon,c)}(W;\one)\ar^{\iota_\epsilon^{c,\infty}}[d]\ar^{\delta_1}[r]& \H(W,\p W)\ar^-{\iota_{-\infty}^{\epsilon,c}}[r]\ar^{\iota_{-\infty}^{\epsilon,\epsilon}=\id}[d] & \SH^{(-\infty,c)}(W;\one)\ar^{\iota_{-\infty}^{c,\infty}}[d]\ar[r] & \cdots\\
\cdots \ar[r] & \SH^{+}(W;\one)\ar[r]^{\delta_2} & \H(W,\p W)\ar^-{\iota_{-\infty}^{\epsilon,\infty}}[r] & \SH(W;\one)\ar[r]&\cdots
}
\]
For any $c>c_{\SH}(W)$, we have $\iota_{-\infty}^{\epsilon,c}=0$. Thus $\delta_1$ and also $\iota_{\epsilon}^{c,\infty}$ are nonzero. This shows that $c_{\SH^+}(W)\leq c$, and hence $c_{\SH}(W)\geq c_{\SH^+}(W)$.

Suppose $\SH(W;\one)=0$ and $\rk\H(W,\p W)=1$. Then $\iota_{-\infty}^{\epsilon,c}= 0$ if and only if $\delta_1\neq 0$, and this  is equivalent also to $\iota_\epsilon^{c,\infty}\neq0$ since $\delta_2$ is an isomorphism. 
\end{proof}

Before showing the announced results about $c_{\SH^+}(W)$ in Section \ref{sec:proofs}, we prove Theorem \ref{t:cgamma}, which asserts $c_{\sigma}(W)\leq c_\SH(W)$.

\subsection*{Proof of Theorem \ref{t:cgamma}}
It suffices to show that $\sigma(H)\leq c_{\SH}(W)$ for every smooth Hamiltonian $H\colon S^1\x W\to\R$ with support in $S^1\x (W\setminus\p W)$ when $c_{\SH}(W)$ is finite. For any $a>c_{\SH}(W)$ not belonging to  $\mathrm{spec}(R_\lambda)\cup\mathrm{spec}(H)$, which is a closed nowhere dense set, we extend $H$ to a smooth function $\widetilde H\colon S^1\times \widehat W\to\R$ in the same manner as in defining $\widehat H$ in \eqref{eq:fh_cpt_supp} but with $h'(r)=\epsilon$ replaced by $h'(r)=a$. The homomorphisms in \eqref{e:fa} can be completed to a commutative diagram
\[
\xymatrix@R-1.5pc{
&\HF(\widetilde H;\one)\ar[r]^-{\widetilde\pi_a}&\HF^{(a,\infty)}(\widetilde H;\one)
\\
\HF(f;\one)\ar[dr]^{\Phi_{f,H}}\ar[ur]^{\Phi_{f,\widetilde H}}&&
\\
&\HF(H;\one)\ar[r]^-{\pi_a} \ar[uu]_{} &\HF^{(a,\infty)}(H;\one)\ar[uu]_{\Phi}^\cong
}
\]
where vertical arrows are continuation homomorphisms, and $\Phi$ is even an isomorphism since $H$ and $\widetilde H$ have the same 1-periodic orbits in the action window $(a,\infty)$. We claim 
\[
\Phi_{f,\widetilde H}(e_f)=0.
\]
Once the claim is verified, the diagram shows that $\pi_a\circ\Phi_{f,H}(e_f)=0$ which implies $\sigma(H)\leq a$ and hence $\sigma(H)\leq c_{\SH}(W)$ as we wanted. The claim now is a consequence of the fact that by Lemma \ref{lem:isom} and Remark \ref{r:isom}, there is a commutative diagram
\begin{equation}\label{e:com}
\begin{gathered}
	\xymatrix{
\HF(\widetilde H;\one)\ar[r]^-{\Phi_{\widetilde H,g}}_-{\cong}&\HF(g;\one)\ar[r]^-{}_-\cong&\SH^{(-\infty,a)}(W;\one)\\
&\HF(f;\one)\ar[lu]^{\Phi_{f,\widetilde H}}\ar[r]_-{\cong}\ar[u]^{\Phi_{f,g}}&\SH^{(-\infty,\epsilon)}(W;\one)\ar[u]_{\iota^{\epsilon,a}_{-\infty}}
}
\end{gathered}
\end{equation}
for some function $g\in\mathcal G_a$, where $\HF(g;\one)=\HF^{(-\infty,a)}(g;\one)$ and $\HF(f;\one)=\HF^{(-\infty,\epsilon)}(f;\one)$ since $g$ and $f$ have no one-periodic orbits outside the action windows $(-\infty,a)$ and $(-\infty,\epsilon)$, respectively, by \eqref{e:action_computation}. Notice that the triangular diagram is commutative since the maps involved are continuation maps and that the horizontal arrow $\Phi_{\widetilde H,g}$ is an isomorphism since $\widetilde H$ and $g$ have the same slope. Now, $\iota^{\epsilon,a}_{-\infty}=0$ since $a>c_{\SH}(W)$ and therefore the claim $\Phi_{f,\widetilde H}(e_f)=0$ follows by commutativity of the diagram.
\hfill\qed

\section{Proofs of the main results}\label{sec:proofs}
In this section, we will be working exclusively with contractible loops. Therefore, we will omit the symbol $\one$ from the notation to make formulas more readable and write for instance $\HF(H)$ and $\SH(W)$ for $\HF(H;\one)$ and $\SH(W;\one)$ respectively. We start by proving the following fundamental result, which is an adaptation of \cite[Proposition 5.2]{GG04} to our setting.

\begin{prp}\label{prp:lower_bound}
Let $H\in\widetilde\HH(W,W_\sk)$ and let $a\in(0,-\min H)\setminus\mathrm{spec}(R_\lambda)$. We assume that all elements of the set 
\[
\Gamma:=\{x\in\PP_1^o(H)\mid -\min H<\AA_H(x)<-\min H+a\}
\] 
are nondegenerate. Then there holds
\[
\#\Gamma\geq \rk \Big[\iota_\epsilon^{a,\infty}\colon\SH^{(\epsilon,a)}(W)\to \SH^+(W)\Big],
\]
where $0<\epsilon<\min\mathrm{spec}(R_\lambda)$ as usual.
\end{prp}
\begin{proof}
Let us consider $H$ as in the statement. There is no loss of generality in assuming $-\min H+a\notin\mathrm{spec}(H)$. Indeed since $\mathrm{spec}(R_\lambda)$ is closed, for $a'<a$ sufficiently close to $a$, we have $a'\notin\mathrm{spec}(R_\lambda)$, $\rk \iota_{\epsilon}^{a',\infty}=\rk \iota_{\epsilon}^{a,\infty}$. We choose $\delta>0$ such that 
\[
H|_{W^\delta}=-\min H,\qquad H|_{\widehat W\setminus W^{1-\delta}}=0.
\] 
Let $\widehat H\colon\widehat W\to\R$ be a smooth function such that $\widehat H=H$ on $W^{1-\delta/2}$ and $\widehat H=\widehat h$ on $\widehat W\setminus W^{1-\delta/2}$, where $\widehat h$ is a smooth function depending only on $r$ such that 
\[
\widehat h''\geq0,\qquad \widehat h|_{\widehat W\setminus W} = a(r-1)+c
\] 
for some $c>0$ small enough.
All one-periodic orbits of $\widehat H$ that are not one-periodic orbits of $H$ have action less than $a$ by \eqref{e:action_computation}: 
\[
\{x\in\PP_1(\widehat H)\mid \AA_{\widehat H}(x)>-\min H\}=\{x\in\PP_1(H)\mid \AA_{H}(x)>-\min H\}.	
\]
In order to relate the Floer homology of $H$ with the positive symplectic homology of $W$, we introduce two auxiliary functions. First, we choose a smooth function $k_b\colon\widehat W\to\R$ which is obtained by smoothening the piecewise linear function that is equal to $\min H$ on $W^{\eta}$ for $\eta<\delta$ and to $b(r-\eta)+\min H$ for $b\in\R\setminus\mathrm{spec}(R_\lambda)$ on $W\setminus W^\eta$. The function $k_b$ depends only on $r$ on $\widehat W\setminus W_\sk$ and $k_b''(r)\geq0$, see Figure \ref{fig:sandwich}. Taking $b$ large enough, we have $k_b\geq \widehat H$. The constant one-periodic orbits of $k_b$ have action equal to $-\min H$. We also take $\epsilon$ and $\eta$ small enough so that the following action estimate holds by \eqref{e:action_computation}:
\begin{equation}\label{eq:action_k_b}
-\min H+\epsilon< \AA_{k_b}(x)< -\min H+a \qquad \forall x\in\PP_1(k_b)\setminus\Crit k_b.
\end{equation}

Similarly, we take $f_a\colon\widehat W\to\R$ to be a convex, smooth approximation of the piecewise linear function which is equal to $\min H$ on $W$ and equal to $a(r-1)+\min H$ on $\widehat W\setminus W$, see Figure \ref{fig:sandwich}. We have $f_a\leq \widehat H$. All constant one-periodic orbits of $f_a$ have action $-\min H$ and
\begin{equation}\label{eq:action_f_a}
-\min H+\epsilon< \AA_{f_a}(x)< -\min H+a \qquad \forall x\in\PP_1(f_a)\setminus\Crit f_a.
\end{equation}

\begin{figure}[htb]
\centering
\includegraphics[width=0.7\textwidth,clip]{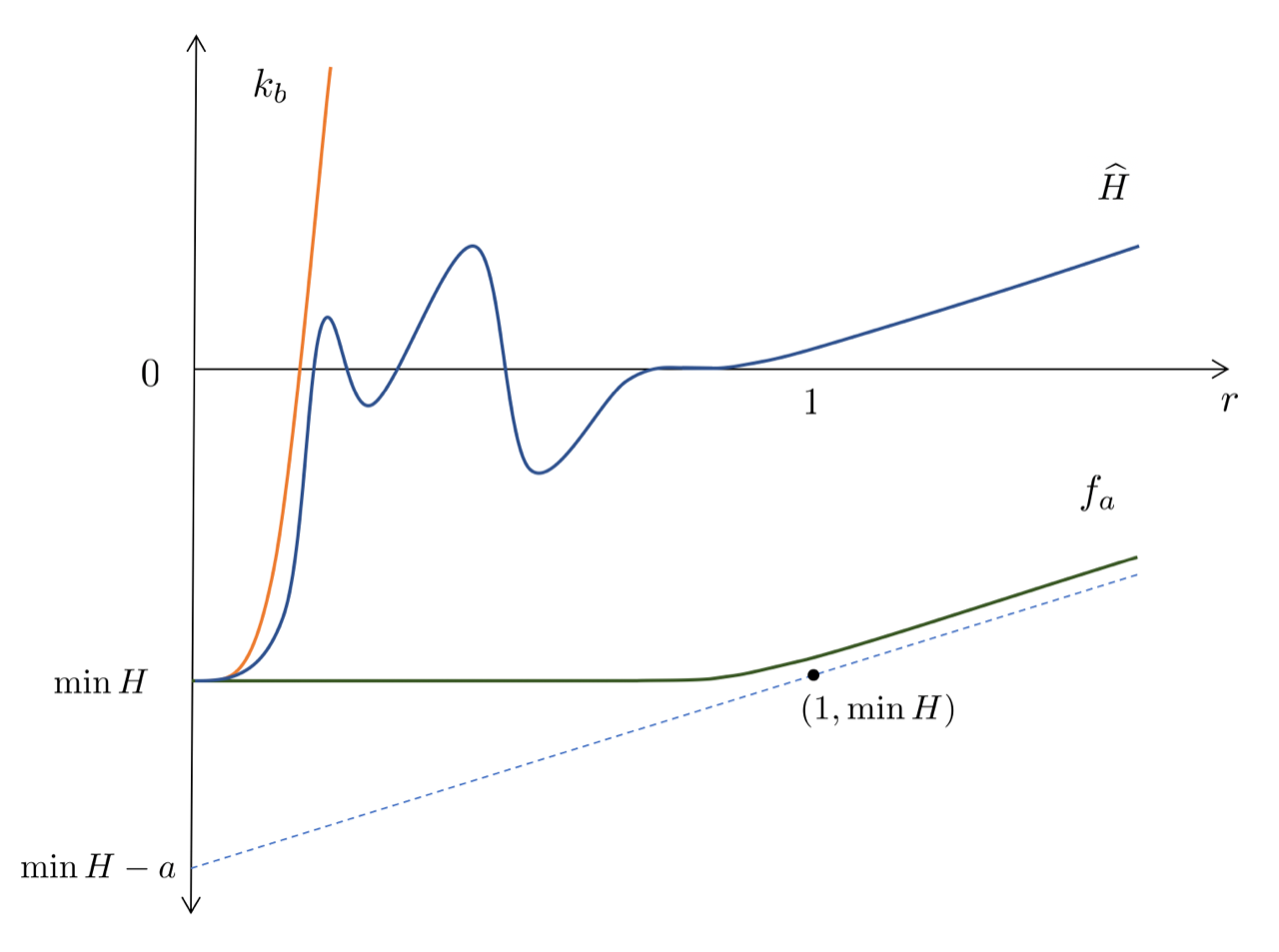}
\caption{The Hamiltonians $\widehat H$, $k_b$ and $f_a$}\label{fig:sandwich}
\end{figure}

We claim that there exists a commutative diagram 
\begin{equation}\label{eq:diagram}
\begin{gathered}
\xymatrix{
& \HF^{(-\min H+\epsilon/2,-\min H+a)}(k_{b})\ar^-{\cong}[r] & \SH^{(\epsilon,b)}(W)\\
\HF^{(-\min H+\epsilon/2,-\min H+a)}(\widehat H)\ar[ru]^{\Phi_2}&&\\
& \HF^{(-\min H+\epsilon/2,-\min H+a)}(f_a)\ar[r]^-{\cong}\ar[lu]_{\Phi_1}\ar[uu]_{\Phi_3} & \SH^{(\epsilon,a)}(W),\ar_{\iota^{a,b}_\epsilon}[uu]
}
\end{gathered}
\end{equation}
where $\epsilon>0$ is chosen so that $-\min H+\epsilon/2\notin\mathrm{spec}(\widehat H)$. 
The maps $\Phi_1$, $\Phi_2$, and $\Phi_3$ are continuation homomorphisms induced by monotone homotopies and satisfy $\Phi_2\circ\Phi_1=\Phi_3$. Once the diagram \eqref{eq:diagram} is established, the proposition follows from
\[
\begin{split}
 \#\Gamma&=\#\big\{x\in\PP_1^o(\widehat H)\mid -\min H<\AA_{\widehat H}(x)<-\min H+a\big\}\\
 &\geq \rk \HF^{(-\min H+\epsilon/2,-\min H+a)}(\widehat H)\\
 &\geq \rk \Big[\iota^{a,b}_\epsilon\colon\SH^{(\epsilon,a)}(W)\to \SH^{(\epsilon,b)}(W)\Big]\\
	&\geq \rk \Big[\iota^{a,\infty}_\epsilon\colon\SH^{(\epsilon,a)}(W)\to \SH^+(W)\Big],
\end{split}
\]
where the second inequality is due to \eqref{eq:diagram}, and the last inequality holds by the identity $\iota_\epsilon^{a,\infty}=\iota_\epsilon^{b,\infty}\circ\iota_\epsilon^{a,b}$.

Let us now define the horizontal isomorphisms in \eqref{eq:diagram} and show that the rectangular diagram commutes. To this purpose, we define 
\[
\widetilde k_b:=k_{b}-\min H-\epsilon/2,\qquad \widetilde f_a:=f_a-\min H-\epsilon/2.
\] 
We also consider a smooth family of functions $\widetilde k^s_b\colon\widehat W\to\R$ for $s\in[0,1-\eta]$ such that\vspace{-5pt}
	\begin{itemize}
	\item $\widetilde k^s_b=-\epsilon/2$ on $W^{\eta+s}$,\vspace{-5pt}
	\item $\widetilde k^s_b(r):=\widetilde k^s_b(r-s)$ for $r>\delta+s$.
	\end{itemize}
We note that $\widetilde k_b^0=\widetilde k_b$ and $\widetilde k_b^{1-\eta}\geq \widetilde f_a$. 
We define the rectangular diagram in \eqref{eq:diagram} as composition of the diagrams:
\begin{equation}\label{eq:Phi_diagram}
\begin{gathered}\xymatrix{
\HF^{(-\min H+\epsilon/2,-\min H+a)}(k_{b}) \ar[r]_-{\Phi_4}^-{\cong} &  \HF^{(\epsilon,\infty)}(\widetilde k_{b})  & \HF^{(\epsilon,\infty)}(\widetilde k_{b}^{1-\eta}) \ar_{\cong}^{\Phi_7}[l] \ar[r]^{\cong}   & \SH^{(\epsilon,b)}(W)\\
\HF^{(-\min H+\epsilon/2,-\min H+a)}(f_a)   \ar[u]_{\Phi_3}  \ar[r]_-{\Phi_5}^-{\cong} & \HF^{(\epsilon,\infty)}(\widetilde f_a) \ar[r]_{\id}\ar[u]_{\Phi_6} & \HF^{(\epsilon,\infty)}(\widetilde f_a) \ar[u]_{\Phi_8} \ar[r]^{\cong} & \SH^{(\epsilon,a)}(W)\ar_{\iota_\epsilon^{a,b}}[u]
}
\end{gathered}
\end{equation}
Since $k_b$ and $f_a$ do not have one-periodic orbits with action greater than $-\min H+a$, 
\[
\begin{split}
&\HF^{(-\min H+\epsilon/2,-\min H+a)}(k_{b})=\HF^{(-\min H+\epsilon/2,\infty)}(k_{b}),\\
&\HF^{(-\min H+\epsilon/2,-\min H+a)}(f_a) =\HF^{(-\min H+\epsilon/2,\infty)}(f_a).
\end{split}
\] 
The maps $\Phi_4$ and $\Phi_5$ are canonical isomorphisms: The functions $k_b$ (resp.~$f_a$) and $\widetilde k_b$ (resp.~$\widetilde f_a$) have the same one-periodic orbits with action shifted by $-\min H-\epsilon/2$ and the same Floer cylinders. These can also be understood as continuation maps of monotone homotopies. The map $\Phi_3$ is a continuation homomorphism, and $\Phi_6$ equals $\Phi_3$ with action shifted by $-\min H-\epsilon/2$. Thus, leftmost rectangle readily commutes.  The monotone homotopy $\widetilde k^s_b$ between $\widetilde k_b^{1-\eta}$ and $\widetilde k_b^0=\widetilde k_b$ has no one-periodic orbit with action equal to $\epsilon$ for all $s$.   Therefore the continuation homomorphism $\Phi_7$ induced by $\widetilde k^s_b$ is an isomorphism. The map $\Phi_8$ is also a continuation homomorphism induced by a monotone homotopy, and the rectangle in the middle commutes since all maps are continuation homomorphisms.  
Finally, the rightmost rectangle follows from \eqref{e:comlem} since $\widetilde f_a$ and $\widetilde k_b^{1-\eta}$ can be taken as $g_a$ and $g_b$ respectively given in Lemma \ref{lem:isom} and, again by action reasons,
\[
\HF^{(\epsilon,\infty)}(\widetilde k_{b}^{1-\eta})=\HF^{(\epsilon,b)}(\widetilde k_{b}^{1-\eta}),\quad \HF^{(\epsilon,\infty)}(\widetilde f_a)=\HF^{(\epsilon,a)}(\widetilde f_a).\qedhere
\]
\end{proof}

\subsection{Proof of Theorem \ref{main_thm}}
The statement is void if $c_{\SH^+}(W)=\infty$. Thus, we suppose  $c_{\SH^+}(W)<\infty$. It is enough to show 
\[
\tilde{c}_\HZ^{\,o}\big (W,W_\sk,c_{\SH^+}(W)\big)\leq c_{\SH^+}(W).
\]
We assume by contradiction that there is $H\in\widetilde{\mathcal H}(W,W_\sk)$ such that 
\[
-\min H>c_{\SH^+}(W),\qquad \AA_H(x)\notin(-\min H,-\min H+c_{\SH^+}(W)]\quad\forall x\in\PP_1^o(H).
\] 
Since $\mathrm{spec}(H)$ is closed and $\mathrm{spec}(R_\lambda)$ is nowhere dense, there exists $a\in(c_{\SH^+}(W),-\min H)$, $a\notin\mathrm{spec}(R_\lambda)$ such that 
\[ 
\AA_H(x)\notin(-\min H,-\min H+a)\qquad\forall x\in\PP_1^o(H).
\] 
This contradicts Proposition \ref{prp:lower_bound}, and thus the theorem is proved.
\qed
\medskip

\subsection{Proof of Theorem \ref{thm:nonautonomous}.(a)}
Let $H\colon S^1\x W\to\R$ be a smooth Hamiltonian with support inside $S^1\times (W\setminus \p W)$ and such that $H|_{S^1\x W_\sk}<0$. 
We extend $H$ smoothly to $\widehat H\colon S^1\x \widehat W  \to\R$ as in the definition of $\HF^{(a,b)}(H)=\HF^{(a,b)}(\widehat H)$ in  \eqref{eq:fh_cpt_supp}. We consider two piecewise linear functions 
\[
\begin{split}
&\bar f\colon\widehat W\longrightarrow \R,\qquad \bar f|_W=-c,\quad\;\bar f|_{\widehat W\setminus W}=\epsilon(r-1)-c \\[.5ex]
&\bar k\colon\widehat W\longrightarrow\R,\qquad \bar k|_{W^{\delta}}=-d,\quad \bar k |_{\widehat W\setminus W^{\delta}}=b(r-\delta)-d
\end{split}
\]
for some positive numbers $c,d,b,\delta>0$ with $b\notin\mathrm{spec}(R_\lambda)$ and for some $0<\epsilon<\min\mathrm{spec}(R_\lambda)$. Smoothening $\bar f\colon\widehat W\to\R$ near $\p W$ and $\bar k\colon\widehat W\to\R$ near $\p W^\delta$, we obtain smooth functions $f\colon\widehat W\to\R$ and $k\colon\widehat W\to\R$ respectively, both of which depend only on $r$ on $\widehat W\setminus W_\sk$ and convex. The assumption on $H$ ensures that for large $b,c>0$ and for small $d,\delta>0$, we have
\[
f(z)\leq \widehat H(t,z)\leq k(z)\qquad \forall (t,z)\in S^1\x\widehat W.
\]
Then for any $a\in(0,d)\setminus\mathrm{spec}(H)$, we have the commutative diagram
\begin{equation}\label{eq:diagram2}
\begin{gathered}\xymatrix@R-1.5pc{
& \HF^{(a,\infty)}(k)\ar^-{\cong}[r] & \SH^{(-\infty,b)}(W)\\
 \HF^{(a,\infty)}(\widehat H)\ar[ru]^{\Phi_2} &&\\
& \HF^{(a,\infty)}(f)\ar[r]^-{\cong}\ar[lu]_{\Phi_1}\ar[uu]_{\Phi_3} & \SH^{(-\infty,\epsilon)}(W), \ar^{\iota_{-\infty}^{\epsilon,b}}[uu]
}
\end{gathered}
\end{equation}
where the $\Phi$'s are continuation homomorphisms induced by monotone homotopies. The rectangular diagram is obtained as the rightmost one in \eqref{eq:Phi_diagram} noticing that $\HF^{(a,\infty)}(k)=\HF(k)$ and $\HF^{(a,\infty)}(f)=\HF(f)$ as $k$ and $f$ do not have one-periodic orbits with action less than $a$ by \eqref{e:action_computation}. The diagram in \eqref{eq:diagram2} readily yields 
\[
\rk \HF^{(a,\infty)}(H)=\rk \HF^{(a,\infty)}(\widehat H)\geq\rk\Phi_3= \rk \iota_{-\infty}^{\epsilon,b} = \rk \iota_{-\infty}^{\epsilon,\infty}
\]
where the last equality holds for large $b>0$. Since $\SH^{(-\infty,\epsilon)}(W)\cong \H(W,\p W)$ by \eqref{eq:SH_epsilon}, the proof is complete.
\qed

\subsection{Proof of Theorem \ref{thm:nonautonomous}.(b)}
Let $H\colon S^1\x W\to(-\infty,0]$ be a smooth Hamiltonian supported in $S^1\times (W\setminus \p W)$ such that $H|_{S^1\x W_\sk}<0$. We extend $H$ to $\widehat H\colon S^1\x\widehat W\to\R$ as in the proof of Theorem \ref{thm:nonautonomous}.(a). Let $k\colon\widehat W\to\R$ be a smooth function obtained also as before by smoothening a piecewise linear function 
\[
\bar k\colon\widehat W\longrightarrow\R,\qquad \bar k|_{W^\delta}=-d,\quad  \bar k|_{\widehat W\setminus W^{\delta}}=\epsilon (r-\delta)-d
\]
for $d,\delta>0$ and $0<\epsilon<\min\mathrm{spec}(R_\lambda)$. We take $d,\delta>0$ small enough and choose a sufficiently large $c>0$ to satisfy
\[
(k-c)(z) \leq \widehat H(t,z) \leq k(z) \qquad \forall (t,z)\in S^1\x\widehat W.
\]
Then for any $a\in(0,d)\setminus\mathrm{spec}(H)$, the following diagram commutes:
\[
\xymatrix@R-1.5pc{
& \HF^{(a,\infty)}(k).\\
\HF^{(a,\infty)}(\widehat H)\ar[ru]^{\Phi_2}&\\
& \HF^{(a,\infty)}(k-c) \ar[lu]_{\Phi_1}\ar[uu]_{\Phi_3}^{\cong}
}
\]
where $\Phi$'s are continuation homomorphisms induced by monotone homotopies. Moreover $\Phi_3$ is an isomorphism since $k$ and $k-c$ have the same slope and possess no one-periodic orbits with action less than $a$. Hence we conclude that $\Phi_2$ is a surjective homomorphism, and this finishes the proof since $\HF^{(a,\infty)}(k)= \HF(k)\cong\H(W;\p W)$.
\qed
\subsection{Proof of Corollary \ref{c:cross}}
Let $Q$ be a closed manifold endowed with a metric $g$ and denote by $D^*Q$ the associated unit-disc cotangent bundle. The foot-point projection $\pi\colon T^*Q\to Q$ gives a bijection between periodic Reeb orbits on $\p( D^*Q)$ with respect to the canonical one-form $\lambda_{T^*Q}$ and closed geodesics on $Q$ where the periods of Reeb orbits correspond to the length of geodesics. In particular if $\ell_\one$ denotes the length of the shortest non-constant, contractible closed geodesic, then 
\[
\min\mathrm{spec}(R_\lambda,\one)=\ell_\one. 
\]
Let $\epsilon>0$ be a positive number smaller than this common value. Consider the square root of the energy functional on the loop space of contractible loops 
\[
\mathcal E\colon\mathcal L_\one Q\to\R,\qquad \mathcal E(x)=\Big(\int_0^1\Vert \dot x(t)\Vert_g^2\,\di t\Big)^{1/2},
\]
where $\Vert\cdot\Vert_g$ is the norm induced by $g$. The functional $\mathcal E$ coincides with the length on the set of geodesics and yields a filtration $\H^{(\epsilon,a)}(\mathcal L_\one Q)$ of the singular homology of the loop space for $a\geq\epsilon$ together with inclusion homomorphisms
\[
\jmath_{\epsilon}^{a,b}\colon\H^{(\epsilon,a)}(\mathcal L_\one Q)\to \H^{(\epsilon,b)}(\mathcal L_\one Q)
\]
for $\epsilon\leq a\leq b$. By the action filtration version of Viterbo isomorphism \cite[Theorem 2.9]{Web06}, we have the commutative diagram
\begin{equation}\label{eq:viterbo_filtered}
\xymatrix{
\SH^{(\epsilon,b)}(D^*Q)\ar[r]^-\cong&\H^{(\epsilon,b)}(\mathcal L_\one Q)\\
\SH^{(\epsilon,a)}(D^*Q)\ar[u]^{\iota_\epsilon^{a,b}}\ar[r]^-\cong&\H^{(\epsilon,a)}(\mathcal L_\one Q).\ar[u]^{\jmath_\epsilon^{a,b}}
}
\end{equation}
By definition of $\epsilon$ we have $\SH^+(D^*Q)=\SH^{(\epsilon,\infty)}(D^*Q)$ and $\H(\mathcal L_\one Q,Q)=\H^{(\epsilon,\infty)}(\mathcal L_\one Q)$. Thus,
\[
c_{\SH^+}(D^*Q)=c(\mathcal E):=\inf\{a>0\ |\ \jmath_{\epsilon}^{a,\infty}\neq0\}.
\]
Since $\ell_\one\leq c_{\SH^+}(D^*Q)$ by constructing a suitable radial Hamiltonian, we are left to show $\ell_\one\geq c(\mathcal E)$ in the two cases mentioned in the statement.
\begin{itemize}
\item Let $(Q,g)$ be a closed, non-aspherical homogeneous space. Since $Q$ is non-aspherical, $\ell_\one$ is finite by the classical Lusternik--Fet theorem and the set of closed, non-constant, contractible geodesics with length $\ell_\one$ is non-empty. By \cite[Theorem 5]{Zil}, the map
\[ \jmath_\epsilon^{\ell_\one+\epsilon,\infty}\colon\H^{(\epsilon,\ell_\one+\epsilon)}(\mathcal L_\one Q)\to \H(\mathcal L_\one Q,Q)
\]
is non-zero for small $\epsilon$ small. Thus, $\ell_\one+\epsilon\geq c(\mathcal E)$ and the result follows letting $\epsilon$ to $0$.
\item Let $(Q,g)$ be a two-sphere with strictly positive Gaussian curvature. Abbondandolo and Mazzucchelli show in Lemma \ref{l:bir=sys} below that there is a continuous path $u\colon[-1,1]\to\{\mathcal E\leq \ell_\one\}$ with $u(-1),u(1)\in Q$ representing a non-trivial element in $\H_1(\mathcal L_\one Q,Q)$. Thus $\jmath_{\epsilon}^{\ell_\one+\epsilon,\infty}\neq0$ for all $\epsilon$ sufficiently small and we conclude that $\ell_\one\geq c(\mathcal E)$.
\end{itemize}

\appendix

\section[The monotonicity of the systole of convex Riemannian two-spheres (by Alberto Abbondandolo  and Marco Mazzucchelli)]{The monotonicity of the systole of convex Riemannian two-spheres \textnormal{(by Alberto Abbondandolo\footnote{Ruhr Universit\"at Bochum, Fakult\"at f\"ur Mathematik, \texttt{\href{mailto:alberto.abbondandolo@rub.de}{alberto.abbondandolo@rub.de}}} and Marco Mazzucchelli\footnote{CNRS, \'Ecole Normale Sup\'erieure de Lyon, UMPA, \texttt{\href{mailto:marco.mazzucchelli@ens-lyon.fr}{marco.mazzucchelli@ens-lyon.fr}}})}}
\label{appendix}

Throughout this appendix, the notion of convexity must be understood in the differentiable sense: A compact three-ball $B\subset\R^{3}$ with smooth boundary is strictly convex when there exists a smooth function $F\colon\R^3\to[0,\infty)$ with positive definite Hessian at every point and such that $\partial B=F^{-1}(1)$. Equivalently, the boundary sphere $M=\partial B$, which will always be equipped with the Riemannian metric $g$ that is the restriction of the ambient Euclidean metric, has strictly positive Gaussian curvature. The systole $\sys(M)>0$ is the length of the shortest closed geodesic of $(M,g)$. The main result of this appendix answers in dimension 3 a question that was posed to us by Yaron Ostrover:

\begin{thm}\label{thm:monotonicity_systol}
Let $B_1\subseteq B_2$ be two compact strictly convex three-balls in $\R^3$ with smooth boundary. Then $\sys(\partial B_1)\leq\sys(\partial B_2)$.
\end{thm}

The main ingredient of the proof is the observation that the systole of positively curved Riemannian two-spheres coincides with the classical Birkhoff min-max, as we will now prove.  Let $(M,g)$ be a Riemannian two-sphere. We denote the energy functional on the $W^{1,2}$ free loop space by 
\begin{align*}
E\colon\Lambda M=W^{1,2}(S^1,M)\to[0,\infty),
\qquad
E(\zeta)=\int_{S^1} \|\dot\zeta(t)\|^2_{g}\diff t.
\end{align*}
Here and in the following, we denote by $S^1=\R/\Z$ the 1-periodic circle.
We consider the unit sphere $S^2\subset\R^3$. For each $z\in[-1,1]$, we denote by $\gamma_z\colon S^1\to S^2$ the parallel at latitude $z$, parametrized as
\begin{align*}
 \gamma_z(t) = \left( \sqrt{1-z^2}\cos(2\pi t) , \sqrt{1-z^2}\sin(2\pi t) , z\right).
\end{align*}
For each continuous map $u\colon[-1,1]\to \Lambda M$ such that $E(u(0))=E(u(1))=0$ there exists a unique continuous map $\tilde u\colon S^2\to M$ such that $u(z)=\tilde u\circ\gamma_z$ for each $z\in[-1,1]$.
We denote by $\UU$ the space of such maps $u$ whose associated $\tilde u$ has degree $1$. The Birkhoff min-max value 
\begin{align*}
\bir(M,g) = \inf_{u\in\UU} \max_{z\in[-1,1]} E(u(z))^{1/2}
\end{align*}
is the length of some closed geodesic of $(M,g)$.

\begin{lem}
\label{l:bir=sys}
On every positively curved closed Riemannian two-sphere $(M,g)$, we have \[\bir(M,g)=\sys(M,g).\]
\end{lem}

\begin{proof}
Let $\gamma\colon S^1\to M$ be a shortest closed geodesic of $(M,g)$ parametrized with constant speed, so that 
$E(\gamma)=L(\gamma)^2 = \sys(M,g)^2$.
A theorem of Calabi--Cao \cite{Calabi:1992aa} implies that $\gamma$ is simple, that is, an embedding $\gamma\colon S^1\hookrightarrow M$. We fix an orientation on $M$, and consider the corresponding complex structure of $(M,g)$. Namely, for every non-zero $v\in T_xM$, the tangent vector  $Jv\in T_xM$ is obtained by rotating $v$ in the positive direction of an angle $\pi/2$. We consider the vector field $\nu(t)=J\dot\gamma(t)$ orthogonal to $\dot\gamma(t)$. Notice that $\nu$ is a parallel vector field, since the complex structure $J$ is parallel. If $K_g$ denotes the Gaussian curvature of $(M,g)$, we have
\begin{align}
\label{e:negative_Hessian}
\diff^2 E(\gamma)[\nu,\nu] = \int_{S^1} \big( \|\nabla_t\nu\|_g^2 -K_g \|\dot\gamma\|_g^2 \|\nu\|_g^2 \big)\,\diff t= - \int_{S^1} K_g \|\dot\gamma\|_g^4\diff t <0.
\end{align}

We now consider Morse's finite dimensional approximation of the free loop space (see, e.g., \cite{Milnor:1963rf}). We fix a positive integer $k$ that is large enough so that $d(\zeta(t_0),\zeta(t_1))<\injrad(M,g)$ for all $\zeta\in\Lambda M$ with $E(\zeta)\leq E(\gamma)=\sys(M,g)^2$ and for all $t_0,t_1\in\R$ with $|t_1-t_0|<1/k$. Here, $d$ denotes the Riemannian distance on $(M,g)$. We consider the open finite dimensional manifold
\begin{align*}
 \Lambda_kM=\big\{ \xx=(x_0,...,x_{k-1})\in M\times ...\times M\ \big|\   d(x_i,x_{i+1})<\injrad(M,g)\ \ \forall i\in\Z_k \big\}.
\end{align*}
Such a manifold admits an embedding
\begin{align*}
\iota\colon\Lambda_kM\hookrightarrow\Lambda M,
\qquad
\iota(\xx)=\gamma_{\xx},
\end{align*}
where each restriction $\gamma_{\xx}|_{[i/k,(i+1)/k]}$ is the shortest geodesic parametrized with constant speed joining $x_i$ and $x_{i+1}$. We denote the restricted energy functional by
\begin{align*}
E_k=E\circ\iota\colon\Lambda_k M\to[0,\infty),
\qquad
E_k(\xx)= k\sum_{i\in\Z_k}d(x_i,x_{i+1})^2.
\end{align*}

Let $\xx:=\iota^{-1}(\gamma)$. We consider the tangent vector $\vv:=(v_0,...,v_{k-1})\in T_{\xx}(\Lambda_kM)$ such that $v_i=\nu(i/k)$ for all $i\in\Z_k$. Inequality~\eqref{e:negative_Hessian} readily implies that $\diff\iota(\xx)\vv$ lies in the negative cone of the Hessian $\diff^2E(\gamma)$, since
\begin{equation}
\label{e:negative_Hessian_2}
\begin{split}
\diff^2E_k(\xx)[\vv,\vv]
&
=
\tfrac{\diff^2}{\diff z^2}\big|_{z=0} E(\iota(\exp_{\xx}(z\vv))\\
&
\leq
\tfrac{\diff^2}{\diff z^2}\big|_{z=0} E(\exp_{\gamma(\cdot)}(z\nu(\cdot)))\\
&
=
\diff^2 E(\gamma)[\nu,\nu] \\
&<0.
\end{split} 
\end{equation}
Here, the exponential map in $\Lambda_kM$ is the one associated with the natural Riemannian metric $g\oplus...\oplus g$.

The complement $M\setminus\gamma$ has two connected components $B_+$ and $B_-$, each one diffeomorphic to a two-ball. The vector field $\nu$ points into one of them, say $B_+$. We define the continuous map 
\[
w\colon[-1/3,1/3]\to \Lambda_kM,
\qquad 
w(z)=\exp_{\xx}(z \epsilon \vv).\]
Notice that $w(0)=\xx$. We fix $\epsilon>0$ small enough so that, for all $z\in(0,1/3]$, the loop $\iota(w(\pm z))$ is entirely contained in the open ball $B_\pm$, and by Equation~\eqref{e:negative_Hessian_2}  we have
\[E_k(w(z))<E_k(w(0))=\sys(M,g)^2,\qquad\forall z\in[-1/3,1/3]\setminus\{0\}.\] 

We now consider the open subspaces $U_+,U_-\subset\Lambda_kM$ given by
\begin{align*}
U_\pm  =  \Lambda_kM \cap (B_\pm\times...\times B_\pm).
\end{align*}
We have $w(\pm1/3)\in U_{\pm}$. The flow $\phi_s$ of the anti-gradient $-\nabla E_k$ is complete in positive time $s$ in the sublevel set $E_k^{-1}([0,\sys(M,g)^2])$. We claim that \[\phi_s(w(\pm1/3))\in U_{\pm},\qquad \forall s\geq0.\] Indeed, assume by contradiction that there exists $s_0>0$ such that $\phi_{s_0}(w(\pm1/3))\in\partial U_{\pm}$, and take $s_0$ to be the minimal such time. If $\yy:=\phi_{s_0}(w(\pm1/3))$, the components of the anti-gradient vector $\zz:=-\nabla E_k(\yy)$ are given by 
\begin{align*}
 z_i= 2(\dot\gamma_{\yy}(\tfrac{i}{k}^+)-\dot\gamma_{\yy}(\tfrac{i}{k}^-)),\qquad \forall i\in\Z_k.
\end{align*}
Since $\yy\in\partial U_{\pm}$, at least one of its components $y_i$ must belong to $\partial B_{\pm}$. Assume that all the $y_i$'s belong to $\partial B_\pm$, and therefore they are of the form $y_i=\gamma(t_i)$ for some $t_i\in S^1$. In this case, we  have $z_i=\lambda_i \dot\gamma(t_i)$ for some $\lambda_i\in\R$; but this is impossible, since it would imply that all the components of $\phi_s(w(\pm1/3))$  belong to $\partial B_\pm$ for all $s\in\R$, and thus that $\phi_s(w(\pm1/3))$ belong to  $\partial U_\pm$ for all $s\in\R$. Therefore at least one component $y_i\in\partial B_{\pm}$ is adjacent to a component in the interior $y_{i-1}\in B_{\pm}$. However, this implies that the vector $z_i$ points inside $B_{\pm}$, and therefore $\phi_{s_0-\delta}(w(\pm1/3))\not\in U_{\pm}$ for all $\delta>0$ small enough, contradicting the minimality of~$s_0$.

We set 
$\delta:=\min\{\injrad(M,g),\sys(M)/(4k)\}$. Since $E_k(\phi_s(w(\pm 1/3)))<\sys(M,g)^2$ for all $s\geq0$, and since $\sys(M,g)^2$ is the smallest positive critical value of $E_k$, we can fix a large enough $s>0$ such that  $E_k(\phi_s(w(\pm1/3)))<\delta^2$. We extend $w$ to a map $w\colon[-2/3,2/3]\to \Lambda_kM$ by setting
\begin{align*}
w(\pm z)=\phi_{(3z-1)s}(w(\pm1/3)),\qquad\forall z\in[1/3,2/3].
\end{align*}
Notice that $w(\pm z)\in U_\pm$ for all $z\in(0,2/3]$, and $E_k(w(\pm2/3))<\delta^2$. We set 
\[\yy^{\pm}=(y_0^{\pm},...,y_{k-1}^{\pm}):=w(\pm2/3).\]
For each $r\in[0,1]$, we define $\yy^{\pm}(r)=(y_0^{\pm}(r),...,y_{k-1}^{\pm}(r))$ by
\begin{align*}
y_i^{\pm}(r):=\exp_{y_0^{\pm}}((1-r)\exp_{y_0^{\pm}}^{-1}(y_i^{\pm})).
\end{align*}
Notice that $\yy^{\pm}(0)=\yy^\pm$, $\yy^{\pm}(r)\in  U_\pm$, and 
\begin{align*}
E_k(\yy^{\pm}(r)) & = k\sum_{i\in\Z_k} d(y_i^{\pm}(r) , y_{i+1}^{\pm}(r))^2 
< 4k^2\delta^2 \leq \sys(M,g)^2,\qquad\forall r\in[0,1],\\
E_k(\yy^{\pm}(1)) & =0.
\end{align*}
We extend $w$ to a continuous map $w\colon[-1,1]\to \Lambda_kM$ by setting
\begin{align*}
w(\pm z)=\yy^{\pm}(3z-2),\qquad\forall z\in[2/3,1].
\end{align*}
Finally, we define $u:=\iota\circ w\colon[-1,1]\to \Lambda M$. Notice that the associated continuous map $\tilde u\colon S^2\to M$ has degree 
$1$; indeed, the preimage $u^{-1}(\gamma(t))$ is a singleton for every $t\in S^1$, and the restriction of $u$ to a neighborhood of $u^{-1}(\gamma)$ is a homeomorphism onto its image. Therefore $u\in\UU$, and
\begin{align*}
\bir(M,g)\leq \max_{z\in[-1,1]} E(u(z))^{1/2} = E(u(0))^{1/2} = \sys(M,g).
\end{align*}
On the other hand, $\bir(M,g)^2$ is a positive critical value of $E$, and therefore 
\[\bir(M,g)\geq \sys(M,g). \qedhere\]
\end{proof}

\begin{proof}[Proof of Theorem~\ref{thm:monotonicity_systol}]
We set $M_i:=\partial B_i$, $i=1,2$. Since the regions $B_1\subset B_2$ are strictly convex, for each $x\in M_2$ there exists a unique $\pi(x)\in M_1$ such that 
\[
\|x-\pi(x)\| = \min_{y\in M_1} \|x-y\|.
\]
The map $\pi\colon M_2\to M_1$ is a 1-Lipschitz homeomorphism with respect to the Riemannian metrics $g_i$ on $M_i$ that are restriction of the ambient Euclidean metric. In particular, for every $W^{1,2}$ curve $\gamma_2\colon S^1\to M_2$, if we denote by $\gamma_1\colon=\pi\circ\gamma_2$ its image in $M_1$, we have
\begin{align*}
\int_{S_1} \|\dot\gamma_2(t)\|^2\diff t \geq \int_{S_1} \|\dot\gamma_1(t)\|^2\diff t 
\end{align*}
We denote by $\UU_1$ and $\UU_2$ the family of maps involved in the definition of the Birkhoff min-max values of $M_1$ and $M_2$ respectively. Notice that $\pi\circ u\in\UU_1$ for all $u\in\UU_2$. Therefore, if we denote the energy of $W^{1,2}$ loops $\gamma\colon S^1\to\R^3$ by
\begin{align*}
E(\gamma)=\int_{S^1} \|\dot\gamma(t)\|^2\,\diff t,
\end{align*}
we have
\begin{align*}
 \bir(M_2)
 =
 \inf_{u\in\UU_2} \max_{z\in[-1,1]} E(u(z))^{1/2}
 \geq
 \inf_{u\in\UU_2} \max_{z\in[-1,1]} E(\pi\circ u(z))^{1/2}
 \geq
 \bir(M_1).
\end{align*}
This, together with Lemma~\ref{l:bir=sys}, implies that $\sys(M_2)\geq\sys(M_1)$.
\end{proof}

\bibliographystyle{amsalpha}
\bibliography{HZ_vs_SH}
\end{document}